\documentclass[a4paper,12pt]{article}

\usepackage[utf8x]{inputenc}
\usepackage{amsmath}
\usepackage{amsfonts}
\usepackage{latexsym}
\usepackage{amsthm}
\usepackage{amssymb,amscd}
\usepackage{xargs}
\usepackage{graphicx}
\usepackage{color}
\usepackage{enumitem}
\usepackage{euscript}
\usepackage{psfrag,url, verbatim, eufrak}

\newtheorem{lemma}{Lemma}[section]

\newtheorem{remark}[lemma]{Remark}

\newtheorem{theorem}[lemma]{Theorem}

\def\om{\omega}

\def\sg{\sigma}
\def\eps{\varepsilon}

\def\phi{\varphi}
\def\A{\mathbb A}
\def\R{\mathbb R}
\def\RR{\mathbb R}
\def\C{\mathbb C}

\def\T{\mathbb T}

\def\Z{\mathbb Z}

\def\ZZ{\mathbb Z}

\def\NN{\mathbb N}

\def\PP{\mathcal P}

\def\CCC{\mathcal C}

\def\DDD{\mathcal D}

\def\NNN{\mathcal N}

\def\cF{\mathcal F}
\def\FF{\mathcal F}

\def\TT{\mathbb T}

\def\I{\mathcal I}
\def\OO{\mathcal O}

\def \SSS{\mathcal S}

\def\D{\mathcal D}

\def\SM{\mathcal {SM}}

\def\HH{\mathcal H}
\def\~{\tilde}

\def\ga{\gamma}

\def\ii{^{-1}}
\def\rr{\rho}

\def\la{\lambda}
\def\Lb{\Lambda}
\def\La{\Lambda}

\def\Ga{\Gamma}

\def\dt{\delta}

\def\wt{\widetilde}
\def\ol{\overline}
\def\wh{\widehat}
\def\kk{\kappa}
\def\eps{\varepsilon}
\def\pa{\partial}

\def\Sr{\mathrm{NR}}
\def\Dr{\mathrm{Res}}

\def\textb{\textcolor{blue}}

\title{A second order expansion of the separatrix map 
for trigonometric perturbations of a priori unstable systems }
\author{
M. Guardia\footnote{Universitat Polit\`ecnica de Catalunya,
mguardia@gmail.com},\ \ \ 
V. Kaloshin\footnote{University of Maryland at College Park,
vadim.kaloshin@gmail.com},\ \ \ 
J. Zhang\footnote{University of Toronto,
jianlu.zhang@utoronto.ca} } 

\begin{document}
\maketitle

\begin{abstract}
In this paper we study a so-called separatrix map introduced 
by Zaslav\-skii-Filonenko \cite{Zaslavskii} and studied by 
Treschev and Piftankin \cite{Treschev98, Treschev02a, 
Pift06, PT}. We derive a second order expansion of this map 
for trigonometric perturbations. In \cite{CastejonK, GuardiaK15}, and  \cite{KaloshinZZ},  
applying the results of the present paper, we describe
a class of nearly integrable deterministic systems with stochastic 
diffusive behavior. 
\end{abstract}

\tableofcontents
\section{Introduction}
The main goal of this paper is to derive a second order 
expansion of a so-called separatrix map for a class of 
nearly integrable systems. In nearly integrable Hamiltonian 
systems with one and a half degree of freedom this map 
was introduced by Zaslavskii and Filonenko in \cite{Zaslavskii}. Shilnikov \cite{Shilnikov}, using a very similar geometric idea, 
studied it in a neighborhood of homoclinic orbits without 
restriction to Hamiltonian structure or closeness to integrability. 
Treschev and Piftankin estimated the error terms in
the traditional version of the separatrix map and studied the
multidimensional situation \cite{Treschev98, Treschev02a, Pift06}.

It is becoming clear that the separatrix map is a powerful 
tool to analyze the dynamics in a neighbourhood of homoclinic 
orbits to a normally hyperbolic invariant manifold and instabilites 
of Hamiltonian systems (see the survey \cite{PT} and the papers 
by Treschev \cite{Treschev04,Treschev12}). For this reason in
\cite{KaloshinZZ}, using results of this paper, we perform 
an indepth analysis of the phenomenon of global instabilities 
in nearly integrable Hamiltonian systems. Usually this 
phenomenon, discovered by Arnold \cite{Arnold64}, 
is called {\it Arnold diffusion}.

The purpose of this paper is to have detailed studies 
of the multidimensional separatrix map, proposed in \cite{Treschev02a}, and to compute higher order expansion of this map with smaller remainder terms.

The main motivation for this work is to study certain 
{\it stochastic diffusive behavior} for nearly integrable 
deterministic systems. Such stochastic behavior was 
conjectured by Chirikov \cite{Chirikov79} in the 1970s.

Recall a basic fact from probability theory, which 
is a non-homogenenous version of Donsker's theorem
(see e.g. \cite{EK}). 
Let $\eta\in \R$, $\sigma(\eta)>0$ and $b(\eta)$ be two 
smooth functions of one variable. 
Fix $\dt>0, \ \eta_0\in \R$  and consider a sequence 
\[
\eta_{n+1}=\eta_n+ \sigma(\eta_n)\dt\, \omega_n+ b(\eta_n)\dt^2, 
\]
where $\omega_n$'s are independent random variables 
taking values $1$ or $-1$ with equal probabilities. 
Then for each $s>0$ and $n_\dt = [s\dt^{-2}]$ as $\dt \to 0$ 
the distribution of $\eta_{n_\dt}$ converges weakly to the distribution 
of the Ito diffusion process starting at $\eta_0$ with the drift $b(\eta)$ 
and the variance $\sigma^2(\eta)$. 

In the framework that we study in this paper, 
$\eta$ is an action variable of some Hamiltonian system.
To study long time evolution of $\eta_t$ we consider 
a separatrix map and study a discrete version 
$\eta_{t_n},\ n>1$. Then the variance $\sigma(\eta)$ 
of the discrete version is given by an associated 
Melnikov function and was computed by Zaslavski
(see (\ref{def:SM:Treshev}) for a more precise claim). 
However, to analyze $\eta_t$ as a diffusion process 
we also need to compute the drift function $b(\eta)$. This is the main purpose
of our paper.

In Appendix \ref{sec:Arnold} we apply our analysis to 
the so-called  {\it generalized Arnold example} and compute 
analogs of the drift $b(\eta)$ and the variance $\sigma(\eta)$. 
In  \cite{KaloshinZZ} we combine  these results with results from
\cite{CastejonK} (see also \cite{GuardiaK15}) and construct a class of
probability measures  
for a class of trigonometric 
perturbations of the generalized Arnold example. It turns 
out that the distributions of a certain action component 
under the associated Hamiltonian flow with respect to 
these measures weakly converge to a stochastic diffusion 
process on the line. 

More precisely, in \cite{KaloshinZZ} for an open class of 
trigonometric perturbations for a time one map of a generalized
Arnold example, which is a $4$-dimensional symplectic map 
$\cF$, we construct a normally hyperbolic invariant lamination 
$\Lb$. Leaves of this lamination are diffeomorphic to 
$2$-dimension cylinders $\A=\T \times \R$. We find 
a coordinate change such that the restriction of $F=\cF|_{\Lb}$ 
to this lamination is a skew-product of the form 
\[
F:\mathbb A \times \{0,1\}^\Z, \qquad 
F(x,\om)=(f_\om(x),\sigma \om),
\] 
where $\om\in \{0,1\}^\Z$, $\sigma$ is the shift, and 
$f_\om(x)$ is an exact area-preserving map of the
$2$-dimensional cylinder $\A$. Loosely speaking, using 
the results of the present paper, we obtain a second order 
expansion of $f_\om$'s in the perturbation parameter. 

In \cite{CastejonK} and \cite{GuardiaK15}
we analyze a class of skew products,
which includes the aforementioned system $F(x,\om)$, 
and prove weak convergence to a stochastic diffusion 
process on the line\footnote{Examples of nearly intergrable 
systems with stochastic diffusive behavior were constructed 
by Marco and Sauzin (see \cite{MarcoS04,Sauzin06}).}.

We emphasize that the previously known results on Arnold 
diffusion for the generalized Arnold example or, more 
generally, for  a priori chaotic systems, or a priori unstable 
systems, or a priori stable systems, establish the existence 
of ``special diffusing'' orbits
(see \cite{Be,BernardKZ11,BolotinT99,BourgainK04,
ChengY04, ChengY09, ChengZ13, Cheng13, 
DelshamsLS00,DelshamsLS06a,DelshamsH09, 
DelshamsLS13,FejozGKR,GelfreichT08,GideaL06,
GuardiaK14,Kaloshin03,KaloshinMV04,KaloshinL08,
KaloshinL08a,KaloshinS12, KaloshinZ12, 
Kaloshin-Levi-Saprykina,MarcoS03, Mather91,
Mather91a,Mather93,Mather96,Mather03,Mather08,
Moeckel:1995,Pift06, Treschev04, Treschev12}). In
\cite{KaloshinZZ}, \cite{CastejonK}, and \cite{GuardiaK15}, using 
the analysis of the present paper, we study deterministic evolution 
of a class of probability measures under the deperministic 
Hamiltonian flow and obtain that the distributions of a certain action component weakly converge to a stochastic 
diffusion process.

\subsection{The separatrix map for the Arnold example}
Since the definition and formulas of the separatrix map are 
rather involved, we start by considering a  simple case: 
{\it the generalized Arnold example.} Later, in Section 
\ref{sec:setup}, we consider more general models. In 
this section we consider Hamiltonians of the form 
$H_0+\eps H_1$ with
\begin{equation}\label{def:ArnoldGeneralized}
 \begin{split}
  H_0(I,p,q)\ \ \ \ &=\ \dfrac{I^2}{2}+\dfrac{p^2}{2}+(\cos q-1), \\
  H_1(I,\varphi,p,q,t)&=\ P(I,p,e^{iq},e^{i\varphi},e^{it}),
 \end{split}
\end{equation}
where the unperturbed $H_0$ defines the product of the rotor and 
the pendulum  and   $P(I, p, e^{iq},e^{i\varphi},e^{it})$ is a real valued 
trigonometric polynomial in $\varphi$ and $t$.  For the classical Arnold 
example $P= (1-\cos q)(\cos\varphi+\cos t)$. 

In this section, we describe and provide ``simple'' formulas for the separatrix 
map for such models. The rigorous definition is provided in Section 
\ref{sec:Treshev} and more accurate formulas for model
\eqref{def:ArnoldGeneralized} are 
derived in Appendix 
\ref{sec:Arnold}.

The Hamiltonian of the pendulum $\frac{p^2}{2}+(\cos q-1)$ has 
a saddle at $(p,q)=(0,0)$ whose stable and unstable 
invariant manifolds coincide forming a figure eight 
(see left picture in Figure \ref{fig:apriori-unstable}). 
In the full phase space, the saddle corresponds to 
the normally hyperbolic invariant manifold  $(p,q)=(0,0)$, 
$I\in[I_-,I_+]$, $\varphi\in\TT$. The union of the  figure eights 
for every value of $I\in[I_-,I_+]$ form the stable and unstable 
manifolds of the normally hyperbolic cylinder. Consider 
a neighborhood of  such set. It can be defined by
\[
 \left|\frac{p^2}{2}+(\cos q-1) \right|<c,\quad I\in[I_-,I_+], \varphi\in\TT
\]
for some small $c>0$. 
\begin{figure}
 \qquad 
\includegraphics[width=14cm]{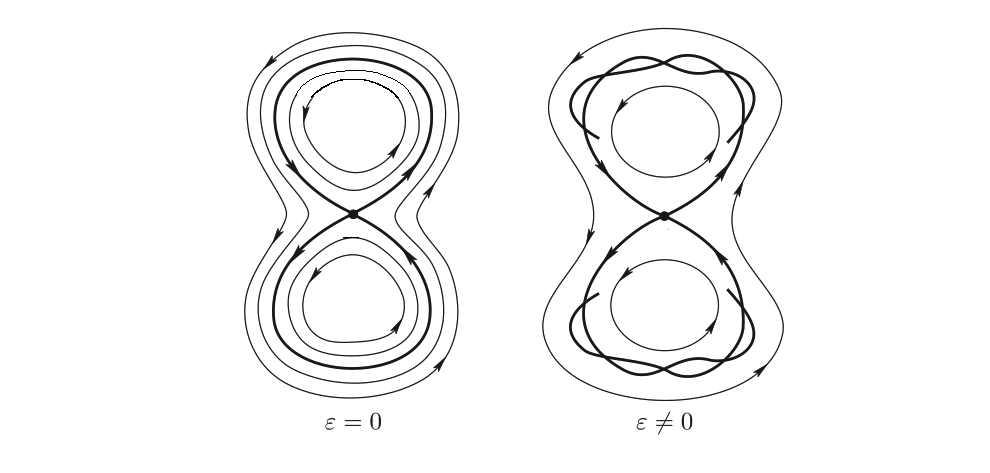}\qquad 
\caption{The figure $8$.}
\label{fig:apriori-unstable}
\end{figure} 
When one adds the perturbation $H_1$, typically, the invariant manifolds of the 
cylinder split. The goal of the separatrix map is to understand the dynamics in 
the neighborhood of the perturbed invariant manifolds, which now intersect 
transversally.

Now we describe the separatrix map. A more  rigorous definition of 
the separatrix map is done in Section \ref{sec:Treshev}. Consider the time 
one map $T_\eps$ associated to the flow of the Hamiltonian $H_\eps$.
We define two fundamental domains, as shown in 
Figure \ref{fig:separatrix-fund-region}, as domains such that 
the image of any point in them under $T_\eps$ does 
not belong to them and some iterate of any point near an unstable manifold
intersects one of these domains.  For any point in the fundamental domains 
one can keep iterating $T_\eps$. If some iterate of such point 
enters again into one of the fundamental domains, we call 
this point the image of the initial point by {\it the separatrix map.} 
Note that some points may not have any of the further iterates 
in the fundamental domains, e.g. those belonging to stable 
invariant manifolds. Then, for these points, we say that 
the separatrix map is not defined. Sometimes, a map of 
this kind is also called {\it an induced map or a return map.}
\begin{figure}
\qquad \qquad 
\includegraphics[width=12.5cm]{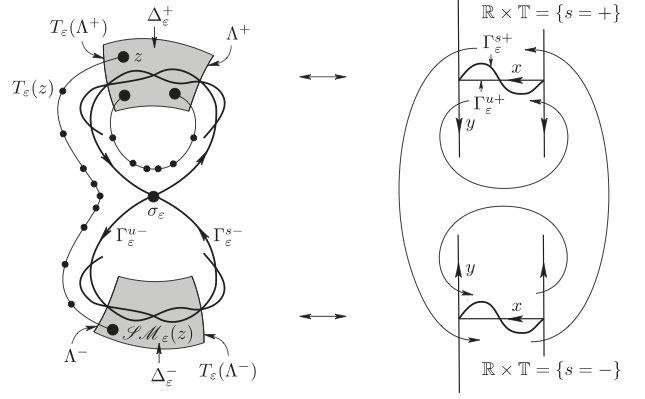}\qquad 
\caption{The fundamental domain}
\label{fig:separatrix-fund-region}
\end{figure} 

To give formulas for the separatrix map, we consider the Fourier expansion of 
the Hamiltonian $H_1(I,\varphi,0,0,t)$,
\[
H_1(I,\varphi,0,0,t)=\sum_{|k|\leq N}H_1^k(I)
e^{2\pi i k\cdot (\varphi,t)},\]
where $N>0$ is the degree of the trigonometric polynomial.

Let $\psi:\RR\longrightarrow [0,1]$  be 
a $\CCC^\infty$ function such that $\psi(r)=0$ for any 
$|r|\geq 1$ and $\psi(r)=1$ for any $|r|\leq 1/2$. We define
\begin{eqnarray*}
\begin{split}
 \ol{\bf H}_1(I,\varphi,t) 
&=\sum_{|k|\leq N}\psi\left(\frac{k\cdot
(I,1)}{\beta}\right) H_1^k(I)e^{2\pi i k\cdot (\varphi,t)}\\
 {\bf H}_1(I,\varphi)&=\ol{\bf H}_1(I,\varphi,0),
 \end{split}
\end{eqnarray*}
for some fixed $\beta>0$ independent of $\eps$ small enough so that for every 
$I\in[I_-,I_+]$ only one harmonic does not vanish (we do not have overlapping 
of resonant zones).

We consider a 
certain system of coordinates $(\eta,\xi,h,\tau,\sigma)$ in the 
fundamental domain (it is explained more precisely in Theorem 
\ref{thm:Treschev}). The last variable $\sigma\in \{-,+\}$ denotes the two 
connected components of the fundamental domain. We define the function
\[
w_0(\eta^*,h^*,\xi^*)=h^*-\frac{(\eta^*)^2}{2}-\eps {\bf 
H}_1(\eta^*,\xi^*).
\]
Then, in  such coordinates, 
the separatrix map is defined for points such that $|w_0|\sim \eps$ and has the 
following (implicit) form 
\begin{equation*}
\begin{aligned}
\eta^*=&\,\eta&&+\eps\pa_\xi\Theta^\sigma(\eta,\xi,\tau) +\eps 
B^{\eta,\sigma} (\eta,\xi, w_0,\tau) +\OO(\eps^{5/3})\\
 \xi^*=&\,\xi &&+\qquad \qquad \qquad \qquad 
 \eta\log\left|\kk^\sigma w_0\right|
+\OO(\eps\log\eps)\\
 h^*=&\,h &&+\eps\pa_\tau\Theta^\sigma(\eta,\xi,\tau) +\OO(\eps^{5/3})\\
 \tau^*=&\,\tau&& + \qquad \qquad \qquad \qquad 
 \log\left|\kk^\sigma w_0\right|+\OO(\eps\log\eps)\\
\sigma^*=&\,\sigma\ \mathrm{sgn}\ w_0
\end{aligned}
\end{equation*}
where $\kk^\sigma>0$ are certain constants (see Section 
\ref{sec:auxiliaryfunctions} and Appendix \ref{sec:Arnold}), 
\[
 B^{\eta,\sigma} (\eta,\xi, w_0,\tau)= 
\frac{1}{\eta}\left({\bf\bar
H}_1(\eta,\xi)-{\bf\bar
H}_1(\eta,\xi+\eta
\log\left|\kk^\sigma w_0\right|)\right)
\]
and $\Theta^\sigma$ are the Melnikov potentials, which are defined as
\[
\Theta^\sigma (\eta,\xi,\tau)=\int_{-\infty}^{+\infty} \left(H_1\left(\Ga^\sigma
(\eta,\xi,\tau+t),t-\tau\right)-H_1\left(\eta,\varphi+\eta 
t,0,0,t-\tau\right)\right)\,dt,
\]
where $\Ga^\sigma$ are the time parameterization of the pendulum separatrices,
that is 
\[
\Ga^\sigma (\eta,\xi,\tau)=\left(\eta, \xi+\eta\tau, 4\arctan
(e^{\sigma\tau}),\frac{2\sigma}{\cosh \tau}\right).
\]
Now we consider a more general set up, proposed by Treschev.

\subsection{The set up}\label{sec:setup}

Consider a 
Hamiltonian system
\begin{equation}\label{def:HamiltonianOriginal}
H_\eps(I,\varphi,p,q,t)=H_0(I,p,q)+\eps H_1(I,\varphi,p,q,t),
\end{equation}
where $I\in \mathbb R^{n}$ are actions,
$\varphi \in \mathbb T^n$ 
are angles and $(p,q)$ belong to an open domain $D\subset\RR^2$. 
Even if not written explicitly, the Hamiltonian $H_1$ may depend on 
the parameter $\eps$. Fix a bounded open set $\D\subset \R^n$.

We assume that for every fixed $I\in\D$, the Hamiltonian $H_0(I,p,q)$ 
has a saddle at $(p,q)=(0,0)$ with two separatrix loops
(see Fig. \ref{fig:apriori-unstable}). In \cite{Treschev02a} it is assumed the
following hypotheses.

\begin{itemize}
 \item[\textbf{H1}] 
The function $H$ is $\CCC^5$-smooth in all arguments while $H_0$ is real-analytic in $p,q$, and $\CCC^5$-smooth in $I$.
\end{itemize}

We consider the  alternative assumption. 
\begin{itemize}

\item[\textbf{H1$'$}] The function $H_0$ is $\CCC^r$
 and $H$ is $\CCC^s$-smooth in all arguments for 
$s\ge 6$ and $r\geq 8s+2$.  
\end{itemize}

That is, we admit lower regularity on $H_0$. On $H_1$ we assume one degree more
of regularity than in \cite{Treschev02a}. It is needed to have  better 
estimates of the separatrix map.
%
%
\begin{itemize}

 \item[\textbf{H2}] For all points $I^0\in\D$, the function $H(I^0,p,q)$ 
has a non-degenerate saddle point at $(p,q)=(p^0,q^0)$ smoothly 
depending on $I$. For all $I^0\in\D$, $(p^0,q^0)$ belongs to 
a connected component of the set 
$$
\{(p,q):\,H_0(I^0,p,q)=H_0(I^0,p^0,q^0)\}.
$$ 
Moreover, 
$(p^0,q^0)$ is the unique critical point of $H_0(I^0,p,q)$ 
in this component.
\end{itemize}

\begin{remark}\label{rmk:saddle}
 Using Prop.1, \cite{Treschev02a}, if one  
assumes that the saddle is at a certain point $(p,q)=(p^0,q^0)$ 
which depends smoothly on $I$, then, one can perform a
symplectic change of coordinates so that the critical point 
is at $(p,q)=(0,0)$ for all $I\in\D$. After such a coordinate 
change $\CCC^r$ in \textbf{H1} is replaced by $\CCC^{r-2}$.   
\end{remark}


We denote the loops of the ``eight''  given by Hypothesis \textbf{H2} by $\ga^\pm (I^0)$. These loops have the natural 
orientation generated by the flow of the system 
(see Fig. \ref{fig:apriori-unstable}). We can define an orientation 
in $D$  by the coordinate system $(p,q)$.

\begin{itemize}
 \item[\textbf{H3}] For all $I^0\in \D$, the natural orientation 
of $\ga^\pm(I^0)$ coincides with the orientation of the domain, 
i. e. the motion of the separatrices is counterclockwise.
\end{itemize}

In \cite{Treschev02a}, Treschev defines the separatrix map 
for Hamiltonians satisfying these hypotheses and obtains a
formula for this map with certain remainder terms. The goal 
of this paper is to refine Treschev formulas in several aspects. 

Note that the formulas for the separatrix map present 
certain differences in what are called {\it a non-resonant} and 
{\it a resonant} regime (see below). Treschev provides global 
formulas for the separatrix map. Here we separate the two 
regimes and give more precise formulas. 
The refinements we do are the following.
\begin{itemize}
\item For the non-resonant regime we compute the separatrix 
map up to {\it 2nd order in $\eps$.} 
\item For the resonant regime we give formulas in slow-fast 
variables and we improve the size of the remainders.
\end{itemize}
We obtain formulas for general Hamiltonians, but also pay 
attention to the particular case of the generalized Arnold example 
\eqref{def:ArnoldGeneralized}. The results for such model are 
presented in Appendix \ref{sec:Arnold}.

\

{\bf Acknowledgement} The authors thank Dmitry Treschev 
for helpful discussions and remarks on the preliminary 
version of the paper. The authors also acknowledge many 
useful discussion with Ke Zhang. The first author is 
partially supported by the Spanish MINECO-FEDER Grant    
 MTM2012-31714 and the Catalan Grant 2014SGR504. 
The second author acknowledges NSF for partial support 
grant DMS-5237860.

\section{The separatrix map: Treschev's results}\label{sec:Treshev}
We devote this section to define the separatrix map and 
 state the results obtained 
in \cite{Treschev02a}.

We want to  define  a separatrix map for points whose 
$(p,q)$-components are ``near'' the unperturbed separatrices 
(see the shaded region on Fig. \ref{fig:figure-eight} below). 
For the unperturbed system  there is a
normally hyperbolic invariant cylinder 
$\Lb=\{p=q=0\}$ and for each $I^0\in \D$ we have invariant 
tori 
\[
\Lb(I^0):=\{I=I^0, p=q=0\}=\Lb\cap \{I=I^0\}.
\]  
These tori are (partially) hyperbolic, i.e there are expanding 
and contracting directions dominating the other ones. 
There exist two asymptotic manifolds,
\[
\begin{split}
\widehat  \Gamma^\pm(I^0)& \subset 
\{(I^0,\varphi,p,q,t):\ \varphi,t \in \T,\ H_0(I^0,p,q)=0
\}\\
\widehat  \Gamma^\pm(I^0)& =\{I^0\} \times \T \times
\gamma^\pm(I_0)\times \T,
\end{split}
\]
where $\gamma^\pm(I_0)$ are the two separatrices in the $(p,q)$ plane. The 
manifolds
$\widehat \Gamma^\pm(I^0)$ consist of unperturbed solutions that approach
$\Lb(I^0)$.

Assume that  $\D$ is an open connected domain with 
compact closure $\overline \D$. In what follows, we consider 
the dynamics of the non-perturbed system in 
a neighbourhood of the unstable
and stable manifolds of the cylinder $\Lb$: 
\[
\widehat   \Gamma= \cup_{I \in \overline \D}
 \left( \widehat  \Gamma^+(I) \cup 
\widehat  \Gamma^-(I)\right).
\]
This neighbourhood contains the most interesting part of 
the perturbed dynamics. It is convenient to pass to the time 
one map $T_\eps$: for any point $(I,\varphi,p,q)$
\[
T_\eps:(I,\varphi,p,q) \longrightarrow (I(1),\varphi(1),p(1),q(1)),
\]
where $(I(t),\varphi(t),p(t),q(t))$ is the solution of 
the Hamiltonian system given by $H_\eps$ with initial conditions 
$(I(0),\varphi(0),p(0),q(0))=(I,\varphi,p,q)$. The map $T_0$ 
has $1$-dimen\-sional hyperbolic tori $L(I) = \pi (\Lb(I))$, 
where the map $\pi: (I,\varphi,p,q,t) \to (I,\varphi,p,q)$ is 
the natural projection. Let $\Sigma^\pm(I) = 
\pi (\widehat  \Gamma^\pm(I))$ be the invariant 
manifolds (see Fig. \ref{fig:figure-eight}). 
We define the separatrix map $\SM_\eps$ corresponding to 
$T_\eps$ in a neighbourhood of the set
\[
\Sigma= \cup_{I \in \overline \D}
 \left( \Sigma^+(I) \cup \Sigma^-(I)\right) = 
\pi (\widehat  \Gamma). 
\]

\begin{figure}
\qquad \qquad \qquad \qquad \qquad \qquad 
\includegraphics[width=5.4cm]{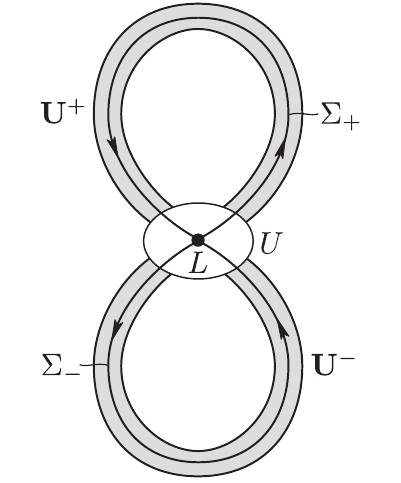}
\caption{Neighbourhoods about the figure $8$.}
\label{fig:figure-eight}
\end{figure}

Let $U$ be a small neighbourhood of the set 
$\cup_ {I \in \overline \D} L(I)$  and let $\bf U$ be 
a neighbourhood of $\Sigma$ (see Figure \ref{fig:figure-eight}, 
where $\Sigma$ is pictured in the case $n=0$). If $\bf U$ is 
sufficiently small, then ${\bf U} \setminus U$ consists of 
two connected components $\bf U^+$ and $\bf U^-$, and, 
thus, $\Sigma^\pm \subset \bf U^\pm \cup U$.

Consider a point $z \in \bf U^+\cup \bf U^-$.
Let $k_1 = k_1(z)$ be the minimal positive integer such 
that $T^{k_1}_\eps(z) \not \in \bf U^+ \cup  \bf U^-$ and let 
$k_2 =k_2 (z)$ be the minimal positive integer such
that $k_2 > k_1$ and $T^{k_2}_\eps(z) \in\bf U^+ \cup  \bf U^-$. 
The trajectory $T^k_\eps(z)$ leaves $\bf U^+ \cup  \bf U^-$ at
$k=k_1$, and it returns to $\bf U^+ \cup  \bf U^-$ at 
$k=k_2$. A point $z$ is said to be good if
$k_2 < +\infty$ and $T^{k_1}_\eps(z),\dots,T^{k_2-1}_\eps(z) \in U$. 
Setting 
\[
{\bf U}_\eps=\{z\in {\bf U^+} \cup  {\bf U^-}: \ z \text{ is good}\}
\]
we obtain maps
\[
\SM_\eps (\cdot,k_2 (\cdot)+k): \ {\bf U}_\eps \to \bf U^+ \cup  \bf U^-,
\]
\[
\SM_\eps (z,k_2 (z)+k) = T^{k_2(z)+k}_\eps(z).
\]
Here $k \in \{0,1,\dots\}$ is a parameter, and it is assumed that
$$
T^{k_2(z)+1}_\eps(z), \dots, 
T^{k_2(z)+k}_\eps(z)\in \bf U^+ \cup  \bf U^-.
$$
In \cite{PT} it is set $t_+ = k_2(z) + k$. 
Since neither $\bf U^+$ nor $\bf U^-$ serves as a fundamental domain, 
there is a considerable freedom for $k$. We would like to avoid 
this freedom. By analogy with the pendulum case (see Fig. 
\ref{fig:separatrix-fund-region}) we let $\Lb^+$ and $\Lb^-$ be 
hyperplanes in $\A \times \A \times \T\ni (I,\varphi,p,q,t)$ 
whose projection onto the $(p,q)$-component are  curves 
going from one of the connected components of the  boundary of $|H_0|<c$ to
another
and transversal to the upper and lower separatrix loops.  
We denote by $\Delta^\pm$ the subdomain of $U_c$ between 
the curves $\Lb^\pm$ and $T_\eps(\Lb^\pm)$ 
(see Fig. \ref{fig:separatrix-fund-region}). We choose such $k$
above that $T^{k_2(z)+k}_\eps(z) \in \Delta^+ \cup \Delta^-$.

To provide formulas for the separatrix map we need to set up some notation.
We define
\begin{equation}\label{def:nu}
 E(I)=H_0(I,0,0),\quad \nu(I)=\pa_I E(I):\D\longrightarrow \RR^n.
\end{equation}
The Hamiltonian $H_1(I,\varphi,0,0,t)$ has the Fourier 
expansion
$$
H_1(I,\varphi,0,0,t)=\sum_{k\in\ZZ^{n+1}}H_1^k(I)
e^{2\pi i k\cdot (\varphi,t)}.$$ 
Let $\psi:\RR\longrightarrow [0,1]$  be 
a $\CCC^\infty$ function such that $\psi(r)=0$ for any 
$|r|\geq 1$ and $\psi(r)=1$ for any $|r|\leq 1/2$. We define
\begin{eqnarray} \label{eq:boldfaceH}
\begin{split}
 \ol{\bf H}_1(I,\varphi,t) 
&=\sum_{k\in\ZZ^{n+1}}\psi\left(\frac{k\cdot
(\nu(I),1)}{\beta}\right) H_1^k(I)e^{2\pi i k\cdot (\varphi,t)}\\
 {\bf H}_1(I,\varphi)&=\ol{\bf H}_1(I,\varphi,0),
 \end{split}
\end{eqnarray}
for some constant $\beta>0$.

To obtain quantitative estimates, we follow \cite{Treschev02a} 
and use skew norms, defined as follows. 
Let $K\subset\RR^m$ be a compact set and $j\in\NN$. 
Then, for functions $f\in\CCC^j(\ol\D\times K)$ we define
\[
 \|f(I,z)\|_j^{(b)}=\max_{0\leq l'+l''\leq
j}b^{l'}\left|\frac{\pa^{l'+l''}f}{\pa^{l'_1}{I_1}\cdots\pa^{l'_n}{I_n}\pa^{
l''_1}{z_1}\cdots\pa^{l''_m}{z_m}} \right|,
\]
where $l''=l_1''+\dots+l_m''$. It is assumed that $f$ can take values 
in $\R^s$, where $s$ is an arbitrary positive integer. The norms 
$\|\cdot \|^{(b)}_r$ are anisotropic, and the variables $r$ play a special 
role in these norms because the additional factor $b$ corresponds to 
the derivatives with respect to $r$. Obviously, $\|\cdot \|^{(1)}_r$ is 
the usual $\CCC^r$-norm. This norm is similar to the skew-symmetric 
norm introduced in \cite{KaloshinZ12}, Section 7.2.

For a function $f \in \CCC^r(\bar {\mathcal D}\times \bf K)$ and 
$g\in \CCC^0(\bar {\mathcal D}\times \bf K)$ we say that
$$
f = \OO^{(b)}_0(g)\ \textup{  if }\ \|f \|_r^{(b)}\le C\,|g|^{k},
$$
where $C$ does not depend on $b$. For brevity we write
\begin{equation} \label{anisotropic-norm}
\|\cdot \|_r^*=\|\cdot \|_r^{(\eps^\dt)},  \ \ \ 
\OO^{(b)} =\OO^{(b)}_1, \ \ \  \OO^*_k=\OO^{(\eps^\dt)}_k.
\end{equation}

First, we state the result obtained in \cite{Treschev02a}. He sets 
$b=\eps^{1/4}$.

\begin{theorem}[\cite{Treschev02a}]\label{thm:Treschev} 
Let conditions {\bf[H1-H3]} hold.
Then, there exist $\CCC^2$ smooth functions 
 \[
  \la,\kk^\pm,\mu^\pm : \ol \DDD\rightarrow \RR,\qquad 
\Theta^\pm : \ol\DDD\times\TT^{n+1}\rightarrow\RR,
 \]
and canonical coordinates $(\eta,\xi,h,\tau)$ such that the following
conditions hold.
\begin{itemize}
 \item $\omega=d\eta\wedge d\xi+dh\wedge d\tau$. 

\item $\eta=I+\OO^{(\eps^{1/4})}(\eps^{3/4}, H_0-E(I)),\ 
\xi+\nu(\eta)\tau=\varphi+f,\ h=H_0+\OO^{(\eps^{1/4})}(\eps^{3/4}, H_0-E(I)),$
where $f$ denotes a function depending only on
$(I,p,q,\eps)$ satisfying $f(I,0,0,0)=0$\footnote
{One can show that $f=\mathcal O^{(\eps^{1/4})}(w^\sg_0+\eps)$.}.

\item Define
\begin{equation}\label{def:omega0}
w_0^\sigma=h^*-E(\eta^*)-\eps {\bf H}_1 (\eta^*,
\xi+\nu(\eta^*)\tau+\mu^\sigma(\eta^*)).
\end{equation}
For any $\sigma=\{+,-\}$ and $(\eta^*,\xi, h^*,\tau)$ such that 
\begin{equation}\label{cond:Treshev}
\begin{split}
c\ii\eps^{5/4}|\log\eps|<|w_0^\sigma|\leq c\eps^{7/8}\quad \\
|\tau|<c\ii\quad c<|w_0^\sigma|\,e^{\la(\eta^*)\ol t}<c\ii,
\end{split}
\end{equation}
the separatrix map $(\eta^*, \xi^*,h^*,\tau^*)=
\SM(\eta, \xi,h,\tau)$ is defined implicitly as follows
\begin{equation}\label{def:SM:Treshev}
\begin{aligned} 
 \eta^*=&\eta-&\eps\pa_\xi\Theta^\sigma(\eta^*,\xi,\tau)\ -\frac{\pa_\xi
w_0^\sigma}{\la}\log\left|\frac{\kk^\sigma w_0^\sigma}{\la}\right|+{\bf O}_2\\
 \xi^*=&\xi+\mu^\sigma+&\eps\pa_{\eta^*}\Theta^\sigma(\eta^*,\xi,\tau)+\frac{ \pa_{\eta^*}
w_0^\sigma}{\la}\log\left|\frac{\kk^\sigma w_0^\sigma}{\la}\right|+{\bf O}_1\\
 h^*=&h-&\eps\pa_\tau\Theta^\sigma(\eta^*,\xi,\tau)\ -\frac{\pa_\tau
w_0^\sigma}{\la}\log\left|\frac{\kk^\sigma w_0^\sigma}{\la}\right|+{\bf O}_2\\
 \tau^*=&\tau+\ol t+&\frac{\pa_{h^*}
w_0^\sigma}{\la}\log\left|\frac{\kk^\sigma w_0^\sigma}{\la}\right|+{\bf O}_1\\
\sigma^*=&\sigma\ \mathrm{sgn}w_0^\sigma,&
\end{aligned}
\end{equation}
where $\la$, $\kk^\pm$ and $\mu^\pm$ are functions of $\eta^*$, $\ol t$ is an
integer such that
\begin{equation} \label{eq:time-interval}
 \left|\tau+\ol t+\frac{\pa_{h^*}
w_0^\sigma}{\la}\log\left|\frac{\kk^\sigma w_0^\sigma}{\la}\right|\right|<c\ii
\end{equation}
and ${\bf O}_1=\OO^{(\eps^{1/4})}(\eps^{7/8})\log^2\eps$, 
${\bf O}_2=\OO^{(\eps^{1/4})}(\eps^{5/4})\log^2\eps$.
\end{itemize}
\end{theorem}

\begin{remark}\label{rmk:InnerAndOuter} Recall that the 
Hamiltonian $H_\eps$ has a normally hyperbolic invariant cylinder 
$\Lambda_\eps$ close to $\Lambda_0=\{p=q=0\}$. 
Dynamics of $H_\eps$ restricted to $\Lambda_\eps$ is often called {\it inner dynamics}. 
Dynamics of orbits belonging to the intersection of $W^s(\Lambda_\eps)$ 
and $W^u(\Lambda_\eps)$ is called {\it outer} dynamics
and thoroughly studied in \cite{DelshamsLS06a, DelshamsLS08}.
\end{remark}

It is helpful to have Figure \ref{fig:separatrix-fund-region} in mind. 

An heuristic explanation of the separatrix map is the following: 
every orbit starting in the fundamental region $\Delta$ 
and sufficiently close to the stable manifold $W^s(\Lb_\eps)$
has three regimes: 

--- approaching the cylinder $\Lambda_\eps$;
 
--- evolution near $\Lambda_\eps$ 

--- departing away from $\Lambda_\eps$ and reaching the fundamental domain.

\vskip 0.1in 

In each of these three regimes we have: 

\vskip 0.08in 

--- Straightening invariant manifolds trivialized the first regime. 

--- The Hamiltonians $w_0^\sigma$ approximate the evolution in the 2nd 
regime. 

--- The splitting potentials $\Theta^\sigma$ describe the third regime with 
an error.

\vskip 0.1in 

It is reasonable to refer to the regime 2)  as inner 
dynamics, since orbits near the cylinder $\Lambda_\eps$ can be shadowed by 
orbits inside of the cylinder.  

It is reasonable to call outer dynamics to the regime 1)+3).
The outer dynamics to the leading order is well described 
by the splitting potentials  $\Theta^\sigma$. 

Look now at the formula (7). The separatrix map has 
contribution from the $w_0^\sigma$, which is the inner dynamics, 
and from $\Theta^\sigma$, which are the splitting potentials.


\subsection{Formulas for auxuliary functions  $\la,\kk^\pm,\mu^\pm$ and 
$\Theta^\sigma$}\label{sec:auxiliaryfunctions}
The functions $\la>0$, $\kk^\pm>0$ and $\mu^\pm\in\RR$ are defined by the
unperturbed Hamiltonian $H_0$ as follows. Hypothesis \textbf{H2} implies that
both eigenvalues of the matrix
\begin{equation}\label{def:MatrixLambda}
\La(I)=
\begin{pmatrix}
 -\pa_{pq}H_0(I,0,0)&-\pa_{qq}H_0(I,0,0)\\
 \ \ \ \pa_{pp}H_0(I,0,0)\ &\pa_{pq}H_0(I,0,0)
\end{pmatrix}
\end{equation}
are real and the trace of this matrix is equal to 0 for all $I$. 
We denote by $\la(I)$ the positive eigenvalue of this matrix. 

We denote by $\ga^\pm(I,\cdot):\RR\longrightarrow \RR^2$ 
the time parameterization of the upper and lower separatrices 
of $H_0$ in the level of energy $H_0(I,p,q)=H_0(I,0,0)$. 
We denote by $a_\pm(I)$ the left
eigenvectors of the matrix $\La(I)$, that is, $ a_+\La=\la a_+$ and $
a_-\La=-\la a_-$,
such that the $2\times 2$ matrix with $a_\pm$ as columns has determinant equal
to one. 

Then, we define
\begin{align}
 \mu^\pm(I)&=\int_{-\infty}^{+\infty}(-\nu(I)+\pa_I H_0
(I,\ga^\pm(I,t)))\,dt\label{def:mu}\\
(\kk^\pm)\ii(I)&=\lim_{t\longrightarrow+\infty}\left[\langle a_+(I),
\ga^\pm(I,-t)\rangle \langle a_-(I),
\ga^\pm(I,t)\rangle e^{2\la(I)t}\right].\label{def:kappa}
\end{align}

To define the functions $\Theta^\sigma$, we have to introduce 
some notation. We define the differential operator
\begin{equation}\label{def:OperatorPartial}
 \pa=\nu(I)\pa_\varphi+\pa_\tau.
\end{equation}
Fix $\beta>0$. The inverse operator $\pa\ii$ is defined on 
the ``resonant'' space
\[
 \mathrm{Res}=\left\{f:\ol\DDD\times\TT^{n+1}: 
f^{k,k_0}=0\ \ \text{ if }\ \ |\langle
k,\nu(I)\rangle+k_0|\leq \beta/2\right\},
\]
where $f^{k,k_0}$ are the Fourier coefficients of $f$. Then, for
$f\in\mathrm{Res}$, we have that $\pa\ii f$ is well defined and satisfies
\[
 \pa\ii f=\OO^*(\beta\ii f).
\]
As we have already mentioned, Treschev in \cite{Treschev02a} 
chooses $\beta=\eps^{1/4}$.

We define
\begin{equation}\label{def:InnerAveraging}
\vartheta^\sigma(I,\varphi,\tau)=-\pa\ii 
\left[H_1(I,\varphi,0,0,\tau)-\ol{\bf
H}_1(I,\varphi,\tau)\right].
\end{equation}
where $\ol{\bf H}_1$ is the Hamiltonian defined in \eqref{eq:boldfaceH}.

Any solution of the unperturbed system lying on 
$\wh \Ga^\pm (I)$  can be written as
\[
 (I,\varphi,p,q)(t)=\Ga^\sigma(I,\varphi,\tau+t)
 \]
with 
\[
 \Ga^\sigma(I,\varphi,\tau)=(I,\varphi+\nu(I)\tau+
\chi^\pm (I,\tau), \ga^\pm (I,\tau))
\]
where $\chi^\pm (I,\tau)$ are solutions of 
\[
 \dot \chi^\pm (I,t)=-\nu(I)+\pa_I H_0(I, \ga^\pm
(I,t)),\qquad\lim_{t\longrightarrow-\infty}\chi^\pm (I,t)=0.
\]
By the definition of $\mu^\sigma$ in \eqref{def:mu}, 
\[
 \mu^\pm (I)=\lim_{t\longrightarrow +\infty}\chi^\pm (I,t).
\]
We define
\begin{equation}\nonumber 
\begin{aligned} 
 H_\pm^\sigma(I,\varphi,\tau,t)&= \\ 
H_1&(\Ga^\sigma(I,\varphi,t),t-\tau)-
H_1(I,\varphi+\nu t+\chi^\sigma (I,\pm\infty),0,0,t-\tau).
\end{aligned}
\end{equation}
Note that $ H_\pm^\sigma(I,\varphi,\tau,t)$ tend to zero 
exponentially as $t\longrightarrow\pm\infty$.

Then,
\begin{equation}\label{def:Melnikov}
\begin{split}
 \Theta^\sigma(I,\varphi,\tau)=&
\vartheta^\sigma (I,\varphi,-\tau)-\vartheta^\sigma
(I,\varphi+\mu^\sigma,-\tau)\\
&-\int_{-\infty}^0 H^\sigma_-(I,\varphi,\tau,t)\,dt
-\int^{+\infty}_0 H_+^\sigma(I,\varphi,\tau,t)\,dt.
\end{split}
\end{equation}
In \cite{Treschev02a}  these functions are called 
\emph{splitting potentials}. The integral term
is the classical Melnikov potential (also often 
called the Poincar\'e function, see for instance \cite{DelshamsG00}).

As we have explained the purpose of this paper is to refine the formulas 
given in Theorem \ref{thm:Treschev}. 


\section{The finite harmonics setting}\label{sec:finiteharmonics}

To get more precise formulas for the separatrix map in the finite 
harmonics setting, we  consider different regions where it is defined. 
These regions  overlap each other. 

First fix some notation. Take a function
$f:\TT^n\times\RR^n\times\RR^2\times\TT\longrightarrow \RR$ with
Fourier series 
\[
f=\sum_{k\in\ZZ^{n+1}}\,f^k(I,p,q)e^{2\pi i k\cdot (\varphi,t)}.
\]
Define $\NNN$ as 
\[
 \NNN(f)=\{k\in\ZZ^{n+1}: f^k\neq 0\}
\]
and 
\[
 \NNN^{(2)}(f)=\{k\in\ZZ^{n+1}: k=k_1+k_2, k_1,k_2\in \NNN(f)\}.
\]
Consider the \emph{non-resonant region}, which stays away from the
resonances created by the harmonics in $\NNN( H_1)\cup\NNN^{(2)}( H_1)$.

Define
\begin{equation}\label{def:SR}
 \Sr_\beta=\left\{I:  \forall k\in \NNN( H_1)\cup\NNN^{(2)}( H_1),\,\,  |k\cdot
(\nu(I),1)|\geq\beta\right\},
\end{equation}
for a fixed parameter $\beta$. The complement of
the non-resonant zone is build up by the different 
resonant zones associated to the harmonics in 
$\NNN( H_1)\cup\NNN^{(2)}( H_1)$. Fix $k\in\NNN( H_1)\cup\NNN^{(2)}( H_1)$, 
then 
we define the resonant zone
\begin{equation}\label{def:DR}
 \Dr^k_\beta=\left\{I: |k\cdot
(\nu(I),1)|\leq\beta\right\}.
\end{equation}
The parameter $\beta$ in both regions will be chosen differently, so that the
different zones overlap.


We abuse notation and we redefine the norms  in
\eqref{anisotropic-norm} as
\[
\|\cdot \|_r^*=\|\cdot \|_r^{(\beta)},  \ \ \ 
\OO^{(b)} =\OO^{(b)}_1, \ \ \  \OO^*_k=\OO^{(\beta)}_k.
\] 
for fixed $\beta>0$. 
Now we can give formulas for the separatrix map in both regions.

\subsection{The separatrix map in the non-resonant regime}
\label{sec:SR}
The main result of this section is Theorem \ref{thm:SM:SR} 
which gives refined formulas for the separatrix map in 
the non-resonant zone 
(see \eqref{def:SR}). To state it we need to define an auxiliary function $w$. 
This $w$ is a slight modification of the functions $w_0^\sigma$ given in 
\eqref{def:omega0}.

Consider a function $g(\eta, r)$. It is obtained 
in Section \ref{sec:Moser} by applying Moser's normal 
form to $H_0$. This function $g$ satisfies 
$g(\eta,r)=\la(\eta) r+\OO(r^2)$, where $\la$ is
the positive eigenvalue of the matrix \eqref{def:MatrixLambda}.
Therefore, $g$ is invertible with respect to the second variable 
for small $r$. Somewhat abusing notation, call $g_r\ii$ 
the inverse of $g$ with respect to the second variable
\footnote{The subindex is to emphasize that the inverse is 
performed with respect to the variable $r$.}. Then, we define the function $w$ 
by
\begin{equation}\label{def:omega}
 w(\eta^*,h^*)=g_r\ii (\eta^*, h^*-E(\eta^*)).
\end{equation}
Note that in the nonresonant regime, $w$ does not depend on whether we are 
close to one of the unperturbed homoclinic loops or the other, as happened in 
Theorem \ref{thm:Treschev}.

\begin{theorem}\label{thm:SM:SR}
Fix $\beta>0$ and $1\ge a>0$. Let conditions 
{\bf [H1$'$,H2,H3]} hold for some  $s\ge 6$, $r\geq 8s+2$.
Then for $\eps$ 
sufficiently small there exist $c>0$ independent of $\eps$ 
and a $\CCC^{s-4}$ canonical coordinates $(\eta,\xi,h,\tau)$  
such that in the non-resonant zone $ \Sr_\beta$ the following 
conditions hold: 
\begin{itemize}
\item the canonical form $\omega=d\eta\wedge d\xi+
dh \wedge d\tau$;

\item $
\eta = I +\OO^*_{1}(\eps)+\OO^*_{2}(H_0-E(I)), \,
\xi + \nu(\eta)\tau=\varphi+f, \,
h=H_0+\OO^*_{1}(\eps)+\OO^*_{2}(H_0-E(I)),$
where $f$ denotes a function depending only on 
$(I,p,q,\eps)$ and such that $f=\mathcal O(w+\eps),\ 
f(I,0,0,0)=0$. 

\item In these coordinates 
$\SM_\eps$ has the following form. 
For any 
$\sigma\in \{-,+\}$ and $(\eta^*,h^*)$ such that 
\[
 c\ii\eps^{1+a}<|w(\eta^*,h^*)|<c\eps,\qquad |\tau|<c\ii, \qquad
c<|w(\eta^*,h^*)|\,e^{\la(\eta^*)\ol t}<c\ii,
\]
the separatrix map $(\eta^*, \xi^*,h^*,\tau^*)=
\SM_\eps(\eta, \xi,h,\tau)$ is defined implicitly as follows
 {\small 
\[
  \begin{split}
 \eta^*=&\ \eta- \ \ \eps
M_1^{\sigma,\eta}+\ \ \eps^2 M_2^{\sigma,\eta}+\ 
\ \OO_3^*(\eps)|\log\eps|\\
   \xi^*=&\ \xi+\ \pa_1 \Phi^\sg (\eta,w (\eta^*,h^*))
+\pa_\eta w(\eta^*,h^*) 
\left[\log |w(\eta^*,h^*)| + \pa_2 \Phi^\sg (\eta^*, w (\eta^*,h^*))
\right] 
\\
&+
\OO_1^*(\eps)\, |\log\eps|\\
 h^*=&\ h - \ \ \eps M_1^{\sigma,h}
+\ \ \eps^2 M_2^{\sigma,h}+\ \ \OO_3^*(\eps)\\
 \tau^*=&\ \tau+\ \qquad \qquad  \ol t \qquad \quad 
+\pa_h w(\eta^*,h^*) 
\left[\log| w(\eta^*,h^*) |+ \pa_2 \Phi^\sg (\eta^*, w (\eta^*,h^*))
\right] 
\\
&+ \OO_1^*(\eps)\,|\log\eps|,
  \end{split}
 \]}
where $w$ is the function defined in \eqref{def:omega}, 
$M^*_i$ and $\Phi^\pm$ are some $\CCC^{s-4}$ functions and 
$\bar t$ is an integer satisfying (\ref{eq:time-interval}). 
The functions $M_i^*$ are evaluated at $(\eta^*, \xi,h^*,\tau)$. 
\end{itemize}

\end{theorem}

This theorem is proven in Section \ref{sec:flowboxcoordinates:SR}.

\begin{remark}
The change of coordinates in the above Theorem is 
$\eps$-close (in the $\CCC^2$-norm) to the system of 
coordinates obtained in Theorem \ref{thm:Treschev}.

The function $M^*_i$ and $\Phi^\pm$ are defined in 
Lemmas \ref{lemma:GluingUnperturb} and \ref{lemma:SM:SR:1}  
respectively,

The functions $\Phi^\sigma$ are the generalizations of 
the functions $\mu^\sigma$ and $\kk^\sigma$. Indeed, 
they satisfy
\[
\pa_\eta\Phi^\sigma(\eta,r)=\mu^\sigma(\eta)+\OO^*_2(r)\ 
\ \text{ and } \ \  e^{\pa_r\Phi^\sigma(\eta,r)}=\kk^\sigma(\eta)+\OO^*_2(r). 
\]
Moreover, the functions $M_i^{\sigma, i}$ satisfy
\[
 M_1^{\sigma, \eta}=\pa_\xi\Theta^\sigma +\OO^*_2(w),
\qquad M_1^{\sigma, h}=\pa_\tau\Theta^\sigma+\OO^*_2(w), 
\]
where $\Theta^\sigma$ is the  splitting potential 
given  in \eqref{def:Melnikov} (compare with the formulas in Theorem
\ref{thm:Treschev}). These last formulas are obtained in Lemma 
\ref{lemma:FirstOrderMs}.
\end{remark}

\begin{remark} \label{rem:non-res-heuristic}
We compare this Theorem with Theorem \ref{thm:Treschev}
and the Remark afterward. Notice that the inner dynamics
in non-resonant zones can be made integrable and essentially 
decoupled from the outer dynamics.
This is reflected in the fact that the $\eta$ and the $h$-components 
have no contribution from $w$ as the inner dynamics is integrable. 

The splitting potential contributions and the integrable contributions 
to rotation of the angular components are essentially the same. 
\end{remark}

Consider the following 
useful example:
\[
H_\eps (I,\varphi,p,q,t)=
\frac{I^2}{2}+\frac{p^2}{2}+(\cos q-1)+\eps H_1(I,\varphi,t).
\]
Notice that the perturbation does not affect the pendulum. 
As the result the perturbed system is the product of 
the pendulum and a perturbed rotor. 
Notice that $\overline {\bf H}_1(I,\varphi,t)=H_1(I,\varphi,t)$ and it is
non-vanishing. Moreover, the splitting potential $\Theta^\sg$ 
vanishes. Since $w^\sg_0$, defined in (\ref{def:omega0}), 
measures the trasition time  $\bar t$ of the separatrix map 
$\SM_\eps$ (see (\ref{cond:Treshev})), it has to be 
independent of angles of the rotor and time. This implies that 
in (\ref{def:omega0}) after substution of $I,p,q,\varphi$ there 
is a cancellation of $\eps {\bf H}_1$ with $\eps$-terms coming 
from $(I-\eta)$ and $(H_0-h)$. This cancellation is not so easy 
to see.

\subsection{The separatrix map in the resonant zones}\label{sec:DR}

To give refined formulas in the resonant zones $\Dr_\beta^k$
(see \eqref{def:DR}) we restrict to Hamiltonian systems 
of two and a half degrees of freedom. Namely,
$I$ and $\varphi$ are one dimensional. 

We consider the resonance 
$(\nu(I),1)\cdot (k_0,k_1)=0$, $(k_0, k_1)\in\NNN(H_1)\cup \NNN^{(2)}(H_1)  
\subset\ZZ^2$, where $\nu(I)$ is the frequency map introduced in 
\eqref{def:nu}. We assume that it is located at $I=0$. Let $A$ be
the variable conjugate to time. Perform a change to 
slow-fast variables. Namely, consider the following change of coordinates
\begin{equation} \label{eq:def-tildeA}
 (J,\theta, D, t)=
\left( \frac{I}{k_0},k_0\varphi+
k_1t,A-\frac{k_1}{k_0}I,
t\right)
\end{equation}
which is symplectic. We obtain the following Hamiltonian
\[
 \wt H(J,\theta,p,q,t)=\wt  H_1\left(k_0I,k_0\ii (\theta-k_1 t), p,q,t\right).
 \]
Note that in $\Dr_\beta^{(k_0,k_1)}$ and in  slow-fast variables, the 
Hamiltonian $\ol{{\bf H}}_1$  is time independent and thus  
$\ol{{\bf H}}_1={{\bf H}}_1$. It is defined as
\[
 \ol{{\bf H}}_1(J,\theta)=\sum_{\ell\in\ZZ}\wt H^{\ell}(J)e^{2\pi i 
\ell\theta}=\sum_{\ell\in\ZZ} H^{\ell(k_0,k_1)}(k_0J)e^{2\pi i \ell\theta}.
\]
To state the next theorem, we recall the definitions of $\nu$ in 
\eqref{def:nu} and of  $\la$, $\mu^\sigma$, $\kk^\sigma$ and $\Theta^\sigma$ in 
Section \ref{sec:auxiliaryfunctions}. 
We also consider a slight modification of the functions $w_0^\sigma$ in
\eqref{def:omega0},
\begin{equation}\label{def:omega01}
w_0=h^*-E(\eta^*)-\eps {\bf H}_1 (\eta^*,
\xi^*).
\end{equation}
This definition is implicit since, as shown in Theorem \ref{thm:SM:DR}, $\xi^*$
depends itself on $w_0$. The function $w_0$, as a function of $\eta^*$ and 
$\xi^*$ is independent of $\sigma$. It certainly depends on $\sigma$ if 
considered as a function with respect to the original variables. 
As happened in Theorem \ref{thm:Treschev}, and unlike Theorem \ref{thm:SM:SR}, 
the functions $w_0$ depend on the homoclinic loop
through the functions $\mu^\sigma$ (which appear in the definition of $\xi^*$).

We also define the functions
\begin{equation}\label{def:DR:Bs}
 \begin{split}
B^{\eta,\sigma} (\eta,\xi,w_0^\sigma,&\tau)= 
 -\frac{1}{\nu(\eta)}\left({\bf\bar
H}_1(\eta,\xi+\nu(\eta)\tau)-{\bf\bar
H}_1(\eta,\xi)\right)\\
\frac{1}{\nu(\eta)}&\left({\bf\bar
H}_1(\eta,\xi+\nu(\eta)\tau+\mu^\sigma(\eta))-{\bf\bar
H}_1(\eta,\xi+\frac{\nu(\eta)}{\la(\eta)}\log\left|\frac{\kk^\sigma 
w_0}{\la}\right|+\mu^\sigma(\eta))\right),\\
B^{h,\sigma} (\eta,\xi,&\tau)= 
\ {\bf\bar
H}_1(\eta,\xi+\nu(\eta)\tau+\mu^\sigma(\eta))-{\bf\bar
H}_1(\eta,\xi+\nu(\eta)\tau).
\end{split}
\end{equation}

\begin{theorem}\label{thm:SM:DR}
Fix $\beta>0$ and $1\ge a>0$. Let conditions {\bf [H1$'$, H2,H3]} 
hold for some  $s\ge 6$, $r\geq 8s+2$. Then for 
 $\eps$ sufficiently small 
there exist $c>0$ independent of $\eps$ and a $\CCC^{s-4}$ 
canonical coordinates $(\eta,\xi,h,\tau)$ such that  in 
the resonant zone $ \Dr^k_\beta$ the following conditions hold:
\begin{itemize}
\item the canonical form 
$\om=d\eta\wedge d\xi+ dh\wedge d\tau$;  
\item $\eta = I +\OO^*_1(\eps,H_0-E(I)), \, 
\xi + \nu(\eta)=\varphi+f, \,
h=H_0+\OO^*_1(\eps,H_0-E(I)),$
where $f$ denotes a $\CCC^{s-4}$ function depending only on 
$(I,p,q,\eps)$ and such that $f(I,0,0,0)=0$ and 
$f=\mathcal O(w_0+\eps)$. 

\item In these coordinates 
$\SM_\eps$ has the following form. For any 
$\sigma\in \{-,+\}$ and $(\eta^*,h^*)$ such that 
\[
 c\ii\eps^{1+a}<|w_0(\eta^*,h^*,\xi^*)|<c\eps,\quad |\tau|<c\ii, \quad
c<|w_0(\eta^*,h^*,\xi^*)|\,e^{\la(\eta^*)\ol t}<c\ii,
\]
 the separatrix map
$(\eta^*, \xi^*,h^*,\tau^*)=
\SM(\eta, \xi,h,\tau)$ is defined implicitly as follows
\begin{equation*}
\begin{aligned}
\eta^*=&\eta &&+\eps\pa_\xi\Theta^\sigma(\eta,\xi,\tau)&&+\eps 
B^{\eta,\sigma} (\eta,\xi,h,\tau)&&+\OO^*(\eps^{5/3})\\
 \xi^*=&\xi+\mu^\sigma&&+\frac{ 
\nu}{\la}\log\left|\frac{\kk^\sigma 
w_0^\sigma}{\la}\right|&&+\eps B^{\xi,\sigma} (\eta,\xi,w_0,\tau, \ol 
t)&&+\OO^*(\eps)\\
 h^*=&h&&+\eps\pa_\tau\Theta^\sigma(\eta,\xi,\tau)&&+\eps 
B^{h,\sigma} (\eta,\xi,\tau)&&+\OO^*(\eps^{5/3})\\
 \tau^*=&\tau+\ol t &&+\frac{1}{\la}\log\left|\frac{\kk^\sigma 
w_0^\sigma}{\la}\right|&&+\eps B^{\tau,\sigma}(\eta,\xi,w_0,\tau, \ol 
t)&&+\OO^*(\eps)\\
\sigma^*=&\sigma\ \mathrm{sgn}\,w_0^\sigma
\end{aligned}
\end{equation*}
where
\[
 \begin{split}
 B^{\xi,\sigma} (\eta,\xi,h,\tau,\ol t)&= 
f_1(\eta,\xi,h,\tau)\tau+f_2(\eta,\xi,h,\tau)\log \left|\frac{\kk^\sigma 
w_0}{\la}\right|+f_3(\eta,\xi,h,\tau, \ol t)\,\ol t\\
 B^{\tau,\sigma} (\eta,\xi,h,\tau ,\ol t)&=  
g_1(\eta,\xi,h,\tau)\tau+g_2(\eta,\xi,h,\tau)\log \left|\frac{\kk^\sigma 
w_0}{\la}\right|+g_3(\eta,\xi,h,\tau, \ol t)\,\ol t
\end{split}
\]
for certain $\CCC^{s-4}$  functions  $f_i$, $g_i$ which satisfy 
$f_i,g_i=\OO^*(1)$. Therefore, the functions $B^{z,\sigma}$ satisfy
\[
 B^{\eta,\sigma}, B^{h,\sigma}=\OO^*(1),\qquad B^{\xi,\sigma}, 
B^{\tau,\sigma}=\OO^*(\log\eps).
\]
\end{itemize}
\end{theorem}

Recall that $\nu(0)=0$. Nevertheless, one can easily see that the function $ 
B^{\eta,\sigma}$ is well defined even as $\nu\rightarrow 0$ since it has a well 
defined limit.

This theorem is proven in Section \ref{sec:flowboxcoordinates:DR}.

\begin{remark}
Vanishing the splitting 
potentials $\Theta^\sigma$ in the formula of the separatrix map in the resonant 
zones one has the inner dynamics. Its non-integrability is given by the 
functions $B^{z,\sigma}$. The splitting potentials $\Theta^\sigma$ encode the 
outer dynamics as in the non-resonant setting (see  Remark 
\ref{rmk:InnerAndOuter}).
\end{remark}

We compare this Theorem with Theorem \ref{thm:Treschev} and 
the remark afterward. Notice that both the inner and the outer
dynamics are nontrivial and $\eps$-coupled due to resonant terms.
Since transition time of the separatrix map is $\mathcal O(\log \eps)$
this gives a term of order $\mathcal O(\eps \log \eps)$ in
the action components, which dominates the contribution of the 
Melnikov function. 

Note that in the aforementioned $\eps$-coupling between the inner 
and the outer dynamics vanishes in the generalized Arnold example 
and the Melnikov function gives dominant contribution 
(see Theorem \ref{thm:SM:DR:Arnold}). 


\subsection{Main steps of the proof and structure of the paper}

In this section we sketch the derivation of the separatrix map done 
by  Treschev \cite{Treschev02a}. At each step we refer to  the
section where it is done in this paper. Recall that the separatrix 
map $\SM_\eps$ is a return map to a certain fundamental 
region ${\bf U^-\cup U^+}$ (see Fig. \ref{fig:figure-eight}).  
Its computation consists of six steps: 
\begin{itemize}
\item Moser's normal form for {\it the unperturbed} pendulum 
near the separatrices (see the blue and the yellow part on 
Figure \ref{fig:gluing-maps}, 
left and Section \ref{sec:NormalForms}).

\item Moser's normal form for {\it the perturbed} pendulum in the 
colored part of Figure \ref{fig:gluing-maps}, left  (see Section
\ref{sec:NormalForms}).

\item The transition map from one yellow region to another one 
in the variables given by Moser's normal  form of the unperturbed system 
(see Section \ref{sec:flowestimates}).

\item The regions ${\bf U^+}$ (resp. ${\bf U^-}$) on 
Figure \ref{fig:gluing-maps}, left are overlapping regions in 
the original coordinates (see Figure \ref{fig:gluing-maps}, right).


\item Compute the  gluing maps (Section \ref{sec:gluingmaps}).

\item Computation of the composition of the transition map
in Moser's normal form variables with the  gluing map (Section
\ref{sec:flowboxcoordinates:SR}).
\end{itemize}

\begin{figure}
\qquad \qquad \qquad 
\includegraphics[width=6.4cm]{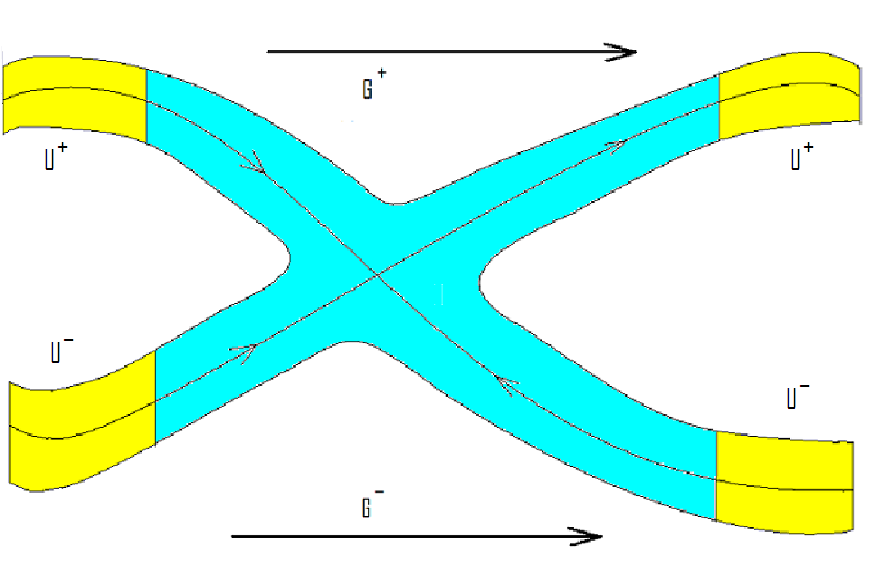}\qquad \ \ \ 
\includegraphics[width=4.2cm]{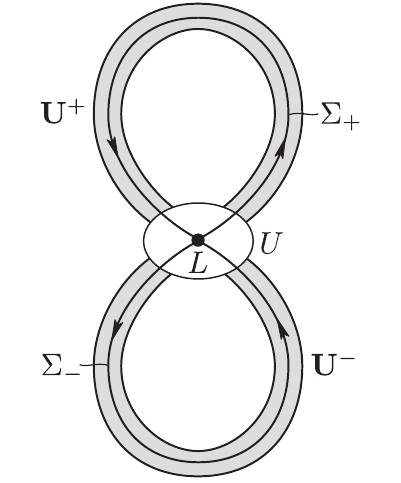}
\caption{Gluing maps.}
\label{fig:gluing-maps}
\end{figure}

\section{Normal forms}\label{sec:NormalForms}
\subsection{Moser normal form  close to the torus}\label{sec:Moser}

We start, as in \cite{Treschev02a}, by performing the classical Moser normal form \cite{Moser56} to the unperturbed 
Hamiltonian $H_0$. To obtain a finitely smooth version of 
this result we apply a result in \cite{BanyagaLW}.

\begin{lemma}\label{lemma:Moser}
Let $H_0$ is $\CCC^r$ satisfying \textbf{[H1$'$]} and \textbf{[H2]}.
For $I\in\D$ and $(p,q)$ 
close to $(0,0)$ there exists a system 
of coordinates $(I,s,x,y)$ such that
\begin{itemize}
\item the change $\FF_0:(I,s,x,y)\longrightarrow (I,\varphi,p,q)$ 
$\CCC^\ell$ smooth and symplectic with $\ell \ge (r-4)/5$.

\item in the variables $H_0\circ \FF_0=\HH_0(I,xy)$ 
is  $\CCC^{\ell+1}$ smooth. We write it as
\[
 \HH_0(I,xy)=E(I)+g(I,xy)
\]
where $g(I,xy)=\la(I)xy+\OO_2(xy)$. 
\end{itemize}
\end{lemma}
\begin{proof}
In Lemma 1,\cite{Treschev02a} the assumption is that $H_0$ is analytic 
in $p,q$. To relax this assumption we use Theorem 1.1 \cite{BanyagaLW}.
Recall that by Remark \ref{rmk:saddle} we can assume that the saddle is at
$(p,q)=(0,0)$ for all $I$ at the expense of loosing two derivatives.

Let $f$ be a $\CCC^{r-2}$ diffeomorphism with the fixed point at the origin 
$f(0)=0$. Let $N$ be a symplectic polynomial map such that $N(0)=0$
and the $k$-jet of $f$ and $N$ coincide at $0$, i.e. 
$D^jf(0)=D^jN(0),\ j=0,\dots,k$. Let $1\le \ell<kA-B$ and $r-2>2k+4$ 
for some integer $\ell$. Then $N$ and $f$ are $\CCC^\ell$ conjugate, i.e. 
there is a $\CCC^\ell$ diffeomorphism  $h$ such that $h^{-1}\circ f \circ h=N$
near the origin, $h(0)=0$.

In our case for each fixed $I$ we have a $2$-dimensional 
symplectic map $f$ with a saddle fixed point. By the remark after Theorem 1.1
of \cite{BanyagaLW}, 
$A=1/4, \ B=1-2A=1/2$. Thus, $1\le \ell < \frac{k-2}{4}<\frac{r-10}{8}$.
\end{proof}

We consider the expansion in $\eps$ of the perturbation of 
the Hamiltonian \eqref{def:HamiltonianOriginal}, namely, 
\[
H_1=H_{11}+\eps H_{12}+\OO(\eps^2)
\]
and define
\begin{equation}\label{def:H1InMoser}
\HH_\eps=\HH_0+\eps
\HH_1+\eps^2\HH_2+\OO(\eps^3)
\end{equation}
with
\[\HH_1(I,s,x,y,t)=H_{11}\circ (\FF_0(I,s,x,y),t), \quad
\HH_{2}(I,s,x,y,t)=H_{12}\circ
(\FF_0(I,s,x,y),t).
\]
Since we assume that $r/8-5/4\geq s$ we have that $\HH^0$ is $\CCC^{s+1}$,
$\HH_1$ is $\CCC^s$ and $\HH_2$ is $\CCC^{s-1}$. Moreover
$\OO(\eps^3)=\OO_{\CCC^{s-2}}(\eps^3)$.

\subsection{Normal forms near the separatrices}\label{sec:normalform}
We extend the Moser normal form to the region 
where $|xy|$ is small (see the shaded/colored region 
Fig. \ref{fig:gluing-maps} on the right/left respectively). 
This normal form applies to both the non-resonant and 
resonant regimes.  Since we want to make a second order 
analysis of the separatrix map we need a more precise 
normal form than in \cite{Treschev02a}. In this normal form
there are two  sources of error.
\begin{itemize}
 \item {\it Expansion in the small parameter $\eps$}: we need 
to perform two steps of normal form instead of one to  
reduce the size of the remainders.
\item {\it Powers of $xy$}: Treschev only performs normal form to 
remove the terms in the perturbation which are independent of 
the product $xy$. He takes
$$\eps^{5/4}|\log\eps|\lesssim|xy|\lesssim \eps^{7/8}.$$ 
We want to remove terms up to the first order in $xy$. Assume that 
\begin{equation}\label{def:sizepq}
\eps^{1+a}\lesssim|xy|\lesssim \eps\qquad\text{with }1\ge  a\geq 0.
\end{equation}
\end{itemize}

We perform the change of coordinates by the Lie Method. First, 
we proceed formally and then we compute the estimates. We consider the expansion of 
the Hamiltonian $\HH$ given in \eqref{def:H1InMoser} and of 
a Hamiltonian of the form $\eps W=\eps W_0+\eps^2 W_1$. Call
$\Phi$ the time-one map associated to the flow of 
$\eps W$. This change is symplectic. Moreover,
\[
\begin{split}
 \HH_\eps\circ \Phi=&\HH_\eps+\eps
\{\HH,W\}+\eps^2\{\{\HH,W\},W\}+\OO\left(\eps^3\right)\\
=&\HH_0+\eps\left(\HH_1+
\{\HH_0,W_0\}\right)\\
&+\eps^2\left(\{\HH_0,W_1\}+\{\HH_1,W_0\}+\{\{\HH_0,W_0\},W_0\}
+\HH_2\right) +\OO\left(\eps^3\right).
\end{split} 
 \]
Now we look for suitable $W_0$ and $W_1$.  


First, compute $W_0$ and then $W_1$.
To  compute $W_0$,  split the Hamiltonian $\HH_1$,
 defined in \eqref{def:H1InMoser}, in the following way 
\[
\begin{split}
 \HH_1(I,s,x,y,t)=&\,\ol {\bf H}_1(I,s,t) +
\HH^{(1)}(I,s,t)+\HH^{(2)}(I,s,y,t)+\HH^{(3)}(I,s,x,t)\\
+ &xy \left(\ol{ \bf H}_2(I,s,t)+ \HH^{(4)}(I,s,t)
+ \HH^{(5)}(I,s,y,t)+  \HH^{(6)}(I,s,y,t)\right)\\+&\OO^*_2(xy)
 \end{split}
\]
where 
\begin{equation} \label{eq:boldface-H2}
 \ol{ \bf H}_2(I,s,t)=\sum_{k\in\ZZ^{n+1}}\psi\left(\frac{k\cdot
(\nu(I),1)}{\beta}\right)\pa_{xy}H_1^{k}(I,0,0)
e^{2\pi i k\cdot (s,t)},
\end{equation}
$\psi$ is the bump function introduced before (\ref{eq:boldfaceH})
and 
\begin{equation}\label{def:HamiltonianSplitting}
\begin{split}
\HH_1^{(1)}(I,s,t)&=\HH_1(I,s,0,0,t)-\ol {\bf H}_1(I,s,t)\\
\HH_1^{(2)}(I,s,t)&=\HH_1(I,s,x,0,t)-\HH_1(I,s,0,0,t)\\
\HH_1^{(3)}(I,s,t)&=\HH_1(I,s,0,y,t)-\HH_1(I,s,0,0,t)\\
\HH_1^{(4)}(I,s,t)&=\pa_{xy}\HH_1(I,s,0,0,t)-
\ol {\bf H}_2(I,s,t)\\
\HH_1^{(5)}(I,s,t)&= \pa_{xy}\HH_1(I,s,x,0,t)-\pa_{xy}\HH_1(I,s,0,0,t)\\
\HH_1^{(6)}(I,s,t)&=
\pa_{xy}\HH_1(I,s,0,y,t)-\pa_{xy}\HH_1(I,s,0,0,t),\\
\end{split}
\end{equation}
where $\ol {\bf H}_1$ is defined in (\ref{eq:boldfaceH}). The functions $\ol
{\bf H}_1$, $\HH_1^{(j)}$, $j=1,2,3$ are $\CCC^s$ whereas the functions
$\HH_1^{(j)}$, $j=4,5,6$ are $\CCC^{s-2}$.

The next lemma contains the first step of the normal form. In the next two 
lemmas we denote by $\{\cdot,\cdot\}_{(x,y)}$, the Poisson bracket with 
respect to the conjugate variables $(x,y)$. 

\

\begin{lemma}
\label{lemma:NormalForm:FirstOrderCohomological}
There exists a $\CCC^{s-2}$ smooth 
solution 
$W_0(I,s,x,y,t)$ 
of the equation 
\[
 \left(\nu(I)+\pa_I g(I,xy)\right)\pa_s W_0+\pa_t W_0+
\{g(I,xy), W_0\}_{(x,y)}+\sum_{j=1}^3\HH_1^{(j)}+
xy \sum_{j=4}^6\HH_1^{(j)}=0.
\]
The functions $W_0$ satisfies
\[
 W_0=\OO^*(\beta\ii),
\]
where $\OO^*$ is defined in (\ref{anisotropic-norm}).
\end{lemma}

\begin{proof}
 We take $W_0=\sum_{j=1}^6W_0^{(j)}$ and solve 
the equations 
 \[
 \begin{split}
 \left(\nu(I)+\pa_I g(I,xy)\right)\pa_s W_0^{j}+\pa_t W_0^{j}+\{g(I,xy),
W_0^{j}\}_{(x,y)}+\HH_1^{(j)}&=0,\ j=1,2,3\ \\
 \left(\nu(I)+\pa_I g(I,xy)\right)\pa_s W_0^{j}+\pa_t W_0^{j}+\{g(I,xy),
W_0^{j}\}_{(x,y)}+xy\HH_1^{(j)}&=0,\ j=4,5,6.
\end{split}
\]
Each equation is solved as follows. For the first and fourth 
ones, we just have 
\[
 \begin{split}
 \left(\nu(I)+\pa_I g(I,xy)\right)\pa_s W_0^{1}+\pa_t
W_0^{1}+\HH_1^{(1)}&=0\\
 \left(\nu(I)+\pa_I g(I,xy)\right)\pa_s W_0^{4}+\pa_t
W_0^{4}+xy\HH_1^{(4)}&=0.
\end{split}
\]
Thus, we invert the operator 
\begin{equation}\label{def:OperatorPartialTilde}
\wt\pa:=(\nu(I)+\pa_I g(I,xy))\pa_s
+\pa_t
\end{equation}
by using the Fourier expansion and inverting for each Fourier
coefficient. Note that this operator and the operator $\pa$ 
in \eqref{def:OperatorPartial} satisfy $\wt\pa-\pa=\OO(xy)$. Moreover, recall
that $\HH_0$ is $\CCC^{s+1}$ and therefore so is $g$. Then, $W_0^{1,4}$ are
$\CCC^{s-2}$
and satisfy
$$W_0^{1}=\pa\ii \HH_1^{(1)}=\OO^*(\beta\ii)
\ \text{ and }\ W_0^{4}=\OO^*(xy\beta\ii).
$$

For the others, we use the characteristics method to obtain 
\[
\begin{split}
 W_0^2&=-\int_{-\infty}^0 \HH_1^{(2)}(I,s+
(\nu(I)+\pa_I g(I,xy))t', ye^{\pa_2 g(I,xy) t'}, t+t')\,dt'\\
W_0^3&=-\int^{+\infty}_0 \HH_1^{(3)}(y,s+
(\nu(I)+\pa_I g(I,xy))t', xe^{-\pa_2 g(I,xy) t'}, t+t')\,dt'\\
W_0^5&=-xy\int_{-\infty}^0 \HH_1^{(5)}(I,s+(\nu(I)+\pa_I g(I,xy))t', ye^{\pa_2 
g(I,xy) t'}, t+t')\,dt'\\
W_0^6&=xy\int^{+\infty}_0 
\HH_1^{(6)}(I,s+(\nu(I)+\pa_I g(I,xy)) t', xe^{-\pa_2 g(I,xy) t'}, t+t')\,dt'.
\end{split}
\]
Thus, they are all $\CCC^{s-2}$.
\end{proof}
Now we solve the second order equation. Define 
\[
\wt \HH_2= \{\HH_1,W_0\}+\{\{\HH_0,W_0\},W_0\}
+\HH_2,
\]
where $\HH_2$ is Hamiltonian defined in \eqref{def:H1InMoser}. Using the equation for $W_0$, given in Lemma
\ref{lemma:NormalForm:FirstOrderCohomological}, 
one has that $\wt \HH_2$ is $\CCC^{s-3}$.

Let $\wt\HH_2^{k}(I,0,0),\ k\in \ZZ^{n+1}$ denote the Fourier 
coefficients of $\HH_2$ in $s$ and $t$. We split $\wt \HH_2$ 
in several terms, as done for $\HH^1$, in the following way 
\[
\begin{split}
\wt \HH_2(I,s,x,y,t)=&\,\ol {\bf H}_3(I,s,t) +
\HH_2^{(1)}(I,s,t)+\HH_2^{(2)}(I,s,y,t)+
\HH_2^{(3)}(I,s,x,t)\\
&+ \OO^*(xy)
 \end{split}
\]
with
\begin{equation} \label{eq:boldface-H3}
 \ol{ \bf H}_3(I,s,t)=\sum_{k\in\ZZ^{n+1}}
\psi\left(\frac{k\cdot
(\nu(I),1)}{\beta}\right)\wt\HH_2^{k}(I,0,0)e^{2\pi i k\cdot
(s,t)}
\end{equation}
and 
\[
\begin{split}
\HH_2^{(1)}(I,s,t)&=\wt\HH_2(I,s,0,0,t)-\ol {\bf H}_3(I,s,t)\\
\HH_2^{(2)}(I,s,t)&=\wt\HH_2(I,s,x,0,t)-\wt\HH_2(I,s,0,0,t)\\
\HH_2^{(3)}(I,s,t)&=\wt\HH_2(I,s,0,y,t)-\wt\HH_2(I,s,0,0,t).
\end{split}
\]
All these terms are $\CCC^{s-3}$.

\

\begin{lemma}
\label{lemma:NormalForm:SecondOrderCohomological}
There exists a $\CCC^{s-3}$ smooth 
solution $W_1(I,s,x,y,t)$
of the equation 
\[
 \left(\nu(I)+\pa_I g(I,xy)\right)\pa_s W_1
+\pa_t W_1+\{g(I,xy), W_1\}_{(x,y)}+
\sum_{j=1}^3\HH_2^{(j)}=0.
\]
Moreover, $W_1=\OO^*(\beta^{-3})$, where $\OO^*$ 
is defined in (\ref{anisotropic-norm}).
\end{lemma}
\begin{proof}
As in the proof of Lemma \ref{lemma:NormalForm:FirstOrderCohomological}, 
we take $W_1=\sum_{j=1}^3W_1^{(j)}$ and solve the equations 
 \[
 \left(\nu(I)+\pa_I g(I,xy)\right)\pa_s W_1^{j}+\pa_t W_1^{j}+\{g(I,xy),
W_1^{j}\}_{(x,y)}+\HH_2^{(j)}=0,\quad j=1,2,3.
\]
Each equation is solved in the same way as the first order in Lemma
\ref{lemma:NormalForm:FirstOrderCohomological}. One can see that $\wt \HH_2$
satisfies $\wt \HH_2=\OO^*(\beta^{-2})$. This implies that 
$W_1=\OO^*(\beta^{-3})$ (one can have more precise bounds for
each $W_1^{j}$).
\end{proof}

Denote by 
\[
\Phi: (\wh I,\wh s,\wh x,\wh y,\wh t)\ \longmapsto \ (I,s,x,y,t) 
\]
the time-one map associated to the flow of the Hamiltonian
$\eps W=\eps W_0+\eps^2 W_1$. This change is $\CCC^{s-4}$ and symplectic. In the
next two lemmas
we analyze the change of coordinates and the transformed Hamiltonian. 

\

\begin{lemma}\label{lemma:NormalFormChangeOfCoordinates}
 The change $\Phi$ is $\CCC^{s-4}$ and satisfies the equations
 \[
 \begin{split}
  I =&\, \wh I+\eps M_1^I+\eps^2  M_2^I+
\OO^*\left(\eps^{3}\beta^{-5}\right)\\
=& \, \wh I+\eps \pa_s W_0+\eps^2
\left(\pa_s W_1+\{\pa_s  W_0,W_0\}\right)
+\OO^*\left(\eps^{3}\beta^{-5}\right)\\
  s =& \, \wh s+\eps M_1^s-\eps^2 M_2^s
  +\OO^*\left(\eps^{3}\beta^{-6}\right)\\
 =& \,\wh s-\eps \pa_I W_0-\eps^2 \left(\pa_I W_1
+ \{\pa_I W_0,W_0\}\right)
+\OO^*\left(\eps^{3}\beta^{-6}\right)\\
  x =& \, \wh x+\eps M_1^x +\eps^2 M_2^x
  +\OO^*\left(\eps^{3}\beta^{-5}\right)\\
  =& \,\wh x+\eps \pa_y W_0+\eps^2 \left(\pa_y W_1
+\{\pa_y  W_0,W_0\}\right)
+\OO^*\left(\eps^{3}\beta^{-5}\right)\\
  y = &\, \wh y+\eps M_1^y
+\eps^2 M_2^y+\OO^*\left(\eps^{3}\beta^{-5}\right)\\
 =& \, \wh y-\eps \pa_x W_0
-\eps^2 \left(\pa_x W_1+\{\pa_x W_0,W_0\}\right)
-\OO^*\left(\eps^{3}\beta^{-5}\right)
 \end{split}
 \]
 Moreover, 
 \[
  M_1^z=\OO^*(\beta\ii),\,\,z=I,x,y \quad \text{ and }
M_1^s=\OO^*(\beta^{-2})
 \]
and 
 \[
  M_2^z=\OO^*(\beta^{-3}),\,\,z=I,x,y \quad \text{ and }
M_2^s=\OO^*(\beta^{-4}).
\]
We also have
\[
 xy=\wh x\wh y+\eps M_1^r+\eps^2 M_2^r+\OO^*\left(\eps^3\beta^{-5}\right),
\]
with
 \[
  M_1^r=\OO^*(\beta\ii),\quad \text{ and }
M_2^r=\OO^*(\beta^{-3}).
 \]

The inverse change is of the same form, that is 
 \[
 \begin{split}
  \wh I =& I-\eps M_1^I
  +\eps^2 \wt M_2^I+\OO^*\left(\eps^{3}\beta^{-5}\right)\\
  \wh s =& s-\eps M_1^s-\eps^2 \wt
M_2^s+\OO^*\left(\eps^{3}\beta^{-6}\right)\\
  \wh x =& x-\eps M_1^x+\eps^2 \wt
M_2^x+\OO^*\left(\eps^{3}\beta^{-5}\right)\\
  \wh y = &y-\eps M_1^y+\eps^2 \wt
M_2^y+\OO^*\left(\eps^{3}\beta^{-5}\right)\\
 \end{split}
 \]
 and 
 \[
 \wh x\wh y=xy-\eps M_1^r+\eps^2\wt  M_2^r+\OO^*(\beta^{-5}\eps^3).
\]
The terms $\wt M_2^z$  satisfy the same estimates as $M_2^z$.
\end{lemma}
\begin{proof}
It is enough to recall that 
$$
\wh I=I
+\eps \{I,W\}+\eps^2 \{\{I,W\},W\}+\ldots.
$$ 
To compute the remainder, one has to 
estimate 
\[
 \{\{\{z,W_0\},W_0\}, W_0\}+\{\{z,W_0\},W_1\}+\{\{z,W_1\},W_0\},\qquad 
z=I,s,x,y.
\]
Using the estimates for $W_0$ and $W_1$, one obtains 
the bounds for the remainder.
\end{proof}

Now we can apply this symplectic change of coordinates to the Hamiltonian 
$\HH_\eps$ given by Lemma \ref{lemma:Moser}.

\begin{lemma}\label{lemma:HamInPQSmall}
The Hamiltonian $\HH_\eps \circ\Phi$ is $\CCC^{s-4}$ and is of  the
following form
\[
\begin{split}
 \HH_\eps \circ\Phi (\wh I,\wh s,\wh x,\wh y,\wh t)=&E(\wh I\,)
+g(\wh I,\wh x \wh y)+ \eps \ol{ \bf H}_1(\wh I,\wh s,\wh t\,)+\eps \wh x\wh y \,
\ol{ \bf H}_2(\wh I,\wh s,\wh t\,)\\
 &+\eps^2 \,\ol{ \bf H}_3(\wh I,\wh s,\wh t\,) +\OO^*\left(\eps^{3}\beta^{-4}
+\eps^{2}\beta^{-2}\wh x\wh y+\eps(\wh x\wh y)^2\right).
\end{split}
\]
\end{lemma}

\section{Transition near the singularity dynamics in the normal form}\label{sec:flowestimates}
We compute the equation associated to the Hamiltonian given in Lemma
\ref{lemma:HamInPQSmall}. We drop the hats to simplify notation. Following 
\cite{Treschev02a}, consider a region for the initial conditions of 
the form
{\small \begin{equation}\label{def:DomainForFlow}
U_*=\left\{ (I,s,x,y):c_*<|x|<c_*\ii,\,c_0^{-1} 
(\eps\beta\ii+|xy|)^2
\log^2(xy)\leq |xy|\leq \kappa_*\right\}
\end{equation}}
and a final time $\ol t$ with
\[
 c_*\leq |y_*|\,e^{\pa_1 g(I^*, \rr)\ol t}\leq c_*^{-1}.
\]
Take, as in \cite{Treschev02a}, $c^*$ and $c_0$
independent of $\eps$ and $\kk_*\sim\eps$.

The Hamiltonian obtained in Lemma \ref{lemma:HamInPQSmall} 
has different formulas in non-resonant and resonant zones. 
We first analyze it in the non-resonant zone and later in 
the resonant one, which are defined  in \eqref{def:SR} and 
\eqref{def:DR} respectively.

\subsection{The non-resonant regime}\label{sec:flowestimates:SR}

Recall that we study trigonometric perturbations. 
In the non-resonant zone \eqref{def:SR}, analyzing 
(\ref{eq:boldfaceH}), (\ref{eq:boldface-H2}), and 
(\ref{eq:boldface-H3})
we have that $\ol{\bf H}_j=0$ for $j=1,2,3$. Therefore, we have the
Hamiltonian 
\begin{equation}\label{def:HamNFSR}
\HH_\eps \circ\Phi (\wh I,\wh s,\wh x,\wh y,\wh t)=E(\,\wh I\,)
+g(\wh I,\wh x\wh y)+ \OO^*\left(\eps^{3}\beta^{-4}
+\eps^{2}\beta^{-2}\wh x\wh y+\eps(\wh x\wh y)^2\right),
\end{equation}
which is $\CCC^{s-4}$ and is integrable  up to order 3. 


\

\begin{lemma}\label{lemma:Flow:SR}
Suppose that for some $(I_*,s_*,x_*,y_*)\in U_*$ 
(see \eqref{def:DomainForFlow})
and 
$\ol t\in\RR$, 
\[
 c_*\leq |y_*|\,e^{\pa_2 g(I,\rr)\ol t}\leq c_*\ii.
\]
where $\rr=|x_*y_*|$.
 
Then,
\[
\begin{split}
  s(\ol t)&=\ s^* +\left(\nu(I^*)+\pa_I
g(I^*,\rr)\right)\ol t+\OO^*\left(\eps^{3}\beta^{-5}+\eps^{2}\beta^{-2}
\rr+\eps\rr^2\right)\log^2\rr\\
 I(\ol t)&=\ I^* +\OO^*\left(\eps^{3}\beta^{-4}+\eps^{2}\beta^{-2}
\rr+\eps\rr^2\right)|\log\rr|\\
 x(\ol t)&=\ x_*e^{-\pa_2 g(I^*,\rr)\ol t}
+\OO^*\left(\eps^{3}\beta^{-4}+\eps^{2}\beta^{-2}
\rr+\eps\rr^2\right)\\
  y(\ol t)&=\ y_*e^{\,\pa_2 g(I^*,\rr)\ol t}\left(1
+\OO^*\left(\eps^{3}\beta^{-4}+\eps^{2}\beta^{-2}
\rr+\eps\rr^2\right)\right)\\
x(\ol t)y(\ol t)&=\ 
x^*y^*+\OO^*\left(\eps^{3}\beta^{-4}+\eps^{2}\beta^{-2}
\rr+\eps\rr^2\right)|\log\rr|.
\end{split}
\]
\end{lemma}
\begin{proof}
The proof of this lemma is a direct consequence of 
the particular form of the equations associated to Hamiltonian \eqref{def:HamNFSR}. Indeed, one can easily see that
\[
 \frac{d}{dt}(xy)=\OO^*\left(\eps^{3}\beta^{-4}
+\eps^{2}\beta^{-2}\wh x\wh y+\eps(\wh x\wh y)^2\right).
\]
Therefore, one can easily see that 
\[
 |x(\ol t)y(\ol t)-x^*y^*|=\OO^*\left(\eps^{3}\beta^{-4}
+\eps^{2}\beta^{-2}\rr+\eps\rr^2\right)|\log\rr|.
\]
Taking this into account, we have
\[
 \dot I=\OO^*\left(\eps^{3}\beta^{-4}
+\eps^{2}\beta^{-2}\rr+\eps\rr^2\right)
\]
which leads to the formula for $I(\ol t)$. Using that $I$ is almost constant,
one can easily deduce the formulas for the other variables.
\end{proof}

\subsection{The resonant regime}\label{sec:flowestimates:DR}
To analyze the resonant regime recall that we focus on the case 
{\it of two and a half degrees of freedom.} Namely, $s\in\TT$ 
and $I\in\RR$. We perform a change to slow-fast variables. 
This leads to a Hamiltonian which is almost a first integral,
namely, its time dependent terms are small. 

Fix $(k_0,k_1)\in\ZZ^2$. Assume that the 
resonance $(\nu(I),1)\cdot (k_0,k_1)=0$, 
$(k_0, k_1)\in\NNN^{(2)}(H_1)\subset\ZZ^2$ is located at $I=0$. 
Call $A$ the variable conjugate to time. Then, the change 
\[
 (J,\theta, D, t)=\left( \frac{I}{k_0},k_0s+k_1t,
A-\frac{k_1}{k_0}I, t\right)
\]
is symplectic. Applying this change, one obtains the 
following Hamiltonian. 
We drop the hats to simplify notations.
\[
 \wt\HH(J,\theta,x,y,t)=\wt E(J)+\wt g(J,xy)+ 
\eps \wt \HH_1(J,\theta, xy)+
\OO^*\left(\eps^{3}\beta^{-4}+\eps^{2}\beta^{-2}
 x y+\eps( x y)^2\right).
\]
where 
\[
 \wt E(J)=E(k_0J)+k_1J,
\]
satisfies 
\[
\pa_J \wt E(0)=0, \ \ \ \wt g(J,xy)=g(k_0J,xy)
\] 
and 
\[
 \wt \HH_1(J,\theta, xy)=\]
\[\sum_{j=-N}^N
\left(\ol{ \bf H}_1^{(jk_0,jk_1)}(k_0J)+ 
xy\, \ol{ \bf H}_2^{(jk_0,jk_1)}(k_0J)+
\eps\, \ol{ \bf H}_3^{(jk_0,jk_1)}(k_0J)\right)e^{2\pi i j\theta}.
\]
We use this system of coordinates to analyze the flow in 
the resonant zones. Recall that by construction $\wt\nu(0)=0$. Since we 
consider $\beta>0$ fixed and in order to avoid cluttering the notation, from 
now on, we do not keep track of the $\beta$ dependence of each estimate. We 
also assume $\rr\lesssim\eps$.



\begin{lemma}\label{lemma:Flow:DR}
Suppose that for some $(J_*,\theta_*,x_*,y_*)\in U_*$ and 
$\ol t\in\RR$, 
\[
 c_*\leq |y_*|e^{\pa_2 g(J_*,\rr)\ol t}\leq c_*\ii.
\]
where $\rr=|x_*y_*|$.
 
Then,
\[
\begin{split}
  \theta(\ol t)=&\theta^* +\left(\nu(J^*)
+\pa_1 g(J^*,\rr)\right)\ol t+\eps F_1(J^*,\theta^*,\ol 
t)+\OO^*\left(\eps^{11/6}\right)\\
 J(\ol t)=&J^* +\eps G_1(J^*,\theta^*,\ol 
t)+\OO^*\left(\eps^{11/6}\right)\\
 x(\ol t)=&x_*e^{-\left(\pa_2 g(J^*,\rr)\ol t+\eps \Phi(J^*,\theta^*,
\ol t)\right)} +\OO^*\left(\eps^{11/6}\right)\\
 y(\ol t)=&y_*e^{\pa_2 g(J^*,\rr)\ol t
+\eps \Phi(J^*,\theta^*, \ol t)}
\left(1 +\OO^*\left(\eps^{11/6}\right)\right)\\
x(\ol t)y(\ol t)=&x^*y^*+\OO^*\left(\eps^{5/2}\right),
\end{split}
\]
where 
\begin{align}
F_1 (J,\theta,\ol t)&=-\frac{\nu'(J)}{\nu(J)}\sum_{k=-N, k\neq 
0}^{N}H_1^k(J)e^{2\pi i 
\theta}\left(\ol t-\frac{e^{2\pi i k\nu(J)\ol t}-1}{2\pi i 
k\nu(J)}\right)\label{def:F1}\\
&+\pa_J H_1(J)\ol t+\sum_{k=-N, k\neq 0}^N\frac{1}{2\pi i k \nu(J)}\pa_J 
H_1(J)e^{2\pi i k \theta}\left(e^{2\pi i k \nu (J)\ol t}-1\right)\notag\\
G_1(J,\theta,\ol t)&=\nu(J)\ii \left(\ol{\bf H}_1(J,\theta)-\ol{\bf 
H}_1(J,\theta+\nu(J)\ol t)\right)\label{def:G1}
 \end{align}
and 
\[
 \Phi(J,\theta,\ol t)=
\int_0^{\ol t}\left(\ol{ \bf
H}_2(J,\theta+\nu(J)t)+\pa_{rJ}g(\rr,J)G_1(J,
\theta,
t)\right)dt.
\]
Thus, $G$ has zero average with respect to $\theta$ and $F$ and $G$ satisfy
\[
 F=\OO^*(\log\eps)\qquad \text{ and }\qquad G=\OO^*(1).
\]
\end{lemma}

\begin{proof}
From Lemma \ref{lemma:HamInPQSmall}, we have the equations
\[
 \begin{split}
\dot \theta=&\ \   \nu(J)+ \pa_Jg(J,xy)+ \eps \pa_J\ol{ \bf
H}_1(J,\theta)+\eps  xy \pa_J\ol{ \bf
H}_2(J,\theta)+\eps^2 \pa_J\ol{ \bf
H}_3(J,\theta)\\
&+\OO^*\left(\eps^{3}\beta^{-4}+\eps^{2}\beta^{-2}
xy+\eps(xy)^2\right).\\
\dot J=&\ -\eps \pa_\theta \ol{ \bf
H}_1(J,\theta)-\eps xy \pa_\theta \ol{ \bf
H}_2(J,\theta)-\eps^2 \pa_\theta \ol{ \bf
H}_3(J,\theta)\\
&+\OO^*\left(\eps^{3}\beta^{-4}+\eps^{2}\beta^{-2}
xy+\eps(xy)^2\right)\\
\dot y=&\ \ \pa_r g(J,xy)y+ \eps  y \ol{ \bf
H}_2(J,\theta)+\OO_2^*(\eps\beta\ii,xy)\\
\dot x=&\ -\pa_r g(J,xy)x+ \eps  x \ol{ \bf
H}_2(J,\theta)+\OO_2^*(\eps\beta\ii,xy).
 \end{split}
\]
We also have 
\[
\frac{d}{dt}(xy)=\OO^*\left(\eps^{3}\beta^{-4}+\eps^{2}\beta^{-2}
xy+\eps(xy)^2\right).
\]
Now we compute estimates for this flow. Call $(I^*,s^*,q^*,p^*)$ to the initial
point. Recall that $\rr=x^*y^*$, that we omit the dependence on $\beta$ and 
that we assume $\rr\lesssim\eps$. Then, 
\begin{equation}\label{eq:EvolutionPQ}
|xy-x^*y^*|=\OO^*\left(\eps^{3}\log\eps\right).
\end{equation}
This implies that the first orders of $\dot \theta$ and $\dot J$ are 
independent of 
$x$ and $y$, since only depend on $\rr$. Indeed, we have
\[
 \begin{split}
\dot \theta=&  \nu(J)+ \pa_Jg(\rr,J)+ \eps \pa_J\ol{ \bf
H}_1(J,\theta)+\OO^*(\eps^2)\\
\dot J=&-\eps \pa_\theta \ol{ \bf
H}_1(J,\theta)+\OO^*(\eps^2).
 \end{split}
\]
These equations can be solved perturbatively in powers of $\eps$. 
We look formally for solutions of the following form and  we take $t^*=0$
\[
 \begin{split}
 \theta(t)&=\theta^* +\nu(J^*)t+\pa_J g(J^*,\rr)t+\eps 
F_1(J^*,\theta^*,t)+\text{h.o.t}\\
 J(t)&=J^* +\eps G_1(J^*,\theta^*,t)+\text{h.o.t}
 \end{split}
\]
Plugging these 
expressions into the equation and recalling that 
$g(J^*,\rr)=\pa_Jg(J^*,\rr)=\pa_J^2g(J^*,\rr)=\OO^*(\eps)$, we obtain the 
following equations
\[
\begin{split}
 \dot F_1=\,&\nu'(J^*)G_1 +\pa_J\ol{ \bf
H}_(J^*,\theta^*+\nu(J^*)t)\\
 \dot G_1=\,&-\pa_\theta \ol{ \bf H}_1(J^*,\theta^*+\nu(J^*)t)
\end{split}
\]
We can first solve the second equation, taking 
\[
   G_1(J^*,\theta^*,\ol t)=-\int_0^{\ol t}\pa_\theta \ol{ \bf 
H}_1(J^*,\theta^*+\nu(J^*)t) \,dt,
\]
which leads to the formula of $G_1$ stated in the lemma.

Plugging the expression of $G_1$ into the $F_1$ equation , we obtain 
$F_1$\footnote{The terms 
$F_1$ does not appear in \cite{Treschev02a}},
\[
\begin{split}
   F_1(J^*,\theta^*,\ol t)=&-\nu'(J^*)\int_0^{\ol
t}G_1(J^*,\theta^*,t)\,dt\\
   &+ \int_0^{\ol t}\pa_J\ol{ \bf
H}_1(J^*,\theta^*+\nu(J^*)t) \,dt.
\end{split}
\]
Now, expanding every term in 
Fourier series and integrating we get the formula for $F_1$.
From their  definition we can deduce that the functions $F_1$ and $G_1$ satisfy
\[
 F_1=\OO^*\left(\log\eps\right),\qquad
G_1=\OO^*\left(1\right).
\]
Analyzing the remainders, one can easily compute the size of the higher order 
terms in the evolution of $J$ and $\theta$ given in Lemma \ref{lemma:Flow:DR}.

Now we analyze the flow for the $x$ and $y$ variables. 
Recall that the conditions on the initial $y_*$ and on the time $\ol t$ imply 
\begin{equation}\label{eq:BoundExponentialTime}
 \frac{c_*^2}{\rr}\leq e^{\pa_2 g(J^*,\rr)\ol t}\leq \frac{1}{\rr c_*^2}.
\end{equation}
Using the almost conservation of $xy$ and the formulas for $\theta(t)$ and 
$J(t)$, we have
\[
 \begin{split}
  \dot y=&\left(\pa_r g(\rr,J(t))+ \eps \ol{ \bf
H}_2(J(t),\theta(t))\right)y+\Theta_1(t)\\
\dot x=&-\left(\pa_r g(\rr,J(t))+ \eps   \ol{ \bf
H}_2(J,\theta(t))\right)x+\Theta_2(t).
 \end{split}
\]
where
$\Theta_i=\OO^*(\eps^2\log\eps)$.

Thus, 
\[
\begin{split}
 x(\ol t)=&x_*e^{-\int_0^{\ol t}\left(\pa_r g(\rr,J(t))+ \eps \ol{ \bf
H}_2(J(t),\theta(t),t)\right)\,dt}\\
&+\int_{\sigma}^{\ol 
t}e^{-\int_\sigma^{\ol t}\left(\pa_r g(\rr,J(t))+ \eps \ol{ \bf
H}_2(J(t),\theta(t),t)\right)\,dt}\Theta_1(\sigma)\,d\sigma.
\end{split}
\]
Using \eqref{eq:BoundExponentialTime},
\[
 x(\ol t)=x_*e^{-\int_0^{\ol t}\left(\pa_r g(\rr,J(t))+ \eps \ol{ \bf
H}_2(J(t),\theta(t),t)\right)\,dt}+\OO^*(\eps^2\log^2\eps).
\]
Now we expand $J(t)$ and $\theta(t)$ with the formulas already obtained. Then, 
\[
x(\ol t)=x_*e^{-\left(\pa_r g(\rr,J^*)\ol t+\eps \Phi(J^*,\theta^*,
\ol t)\right)} +\OO^*(\eps^{11/6})
\]
with
\[
 \Phi(J^*,\theta^*,\ol t)=
\int_0^{\ol t}\left(\ol{ \bf
H}_2(J^*,\theta^*+\nu(J^*)t)+\pa_{rJ}g(\rr,J^*)G_1(J^*,
\theta^*,
t)\right)dt .
\]
Using this estimate and \eqref{eq:EvolutionPQ}, we can deduce the following 
formulas for $y(\ol t)$,
\[
 y(\ol t)=y_*e^{\pa_r g(\rr,J^*)\ol t+\eps \Phi(J^*,\theta^*,
\ol t)}\left(1 +\OO^*(\eps^{11/6})\right).
\]
\end{proof}

\section{The gluing maps}\label{sec:gluingmaps}
We compute the gluing maps. First in Section \ref{sec:gluing:H0}, we consider
the unperturbed Hamiltonian $\HH_0$ given in Lemma \ref{lemma:Moser}. Then, we
consider the full Hamiltonian in the non-resonant setting 
(Section \ref{sec:gluingmaps:SR}) and resonance 
setting (Section \ref{sec:gluingmaps:DR}).

\subsection{Unperturbed gluing maps}\label{sec:gluing:H0}
The unperturbed gluing maps are computed in Section 5 of \cite{Treschev02a}. 
Consider the gluing maps
\begin{equation}\label{def:GluingMaps}
\begin{split}
\SSS^+:&\{|y| \text{ is small}, x>0\}\longrightarrow\{ y>0, |x| \text{ is
small}\}\\
 \SSS^-:&\{|y| \text{ is small}, x>0\}\longrightarrow\{ y<0, |x| \text{ is
small}\}.
\end{split}
\end{equation}
The gluing maps of the unperturbed system must satisfy 
the following properties 
\begin{itemize}
\item[(i)] are symplectic
\item[(ii)] preserve $I$ and $xy$. 
\end{itemize}

\begin{lemma}[\cite{Treschev02a}]\label{lemma:GluingUnperturb}Properties
(i), (ii) and \eqref{def:GluingMaps} imply that the gluing maps 
\[
 (I^\pm, s^\pm,x^\pm,y^\pm)=\SSS^\pm(I,s,x,y)
\]
are of the form
\[
\begin{split}
 I^\pm&=I\\
 s^\pm&=s+\pa_I\Phi^\pm (I,xy)\\
 y^\pm&=x^{-1} e^{-\pa_r\Phi^\pm (I,xy)}\\
 x^\pm&=x^{2}y e^{\pa_r\Phi^\pm (I,xy)}
\end{split}
\]
for some functions $\Phi^\pm$. 
\end{lemma}
In \cite{Treschev02a} it is also shown that
$\pa_I\Phi^\pm
(I,r)=\mu^\pm(I)+\OO(r)$ and $e^{\pa_r\Phi^\pm (I,r)}=\kk^\pm (I)+\OO(r)$. 
Notice that the maps $\SSS^\pm$ are $\CCC^{s}$. 

Moreover, since $\rr=xy$,   $x^\pm=\OO(\rr)$ are always
small. One can consider also the
inverse change for $(x,y)$. Since $xy=x^\pm y^\pm$, it is given by 
\[
\begin{split}
 y&=y^\pm \rr e^{-\pa_r\Phi^\pm (I,\rr)}\\
 x&=(y^\pm)\ii e^{\pa_r\Phi^\pm (I,\rr)}.
\end{split}
\]
So, $y$ also is small since $y=\OO(\rr)$. 

\subsection{Gluing maps in non-resonant zones}\label{sec:gluingmaps:SR}

Now we express these gluing maps in the normal form 
coordinates $(\wh I, \wh s, \wh x, \wh y)$. We  use the formulas 
for the normal form coordinates obtained in 
Section \ref{sec:normalform}. 

The gluing maps  have the following form. Each term can 
be expressed in terms of the Hamiltonian $W$ associated to 
the normal form  variables from
Section \ref{sec:normalform}.
\[
\begin{split}
 \wh I^\pm =&\wh I+\eps (M_1^I\circ\SSS^\pm-M_1^I)+\eps^2
(M_2^I\circ\SSS^\pm-\wt M_2^I)+\OO^*(\eps^3\beta^{-5}) \\
 \wh s^\pm =&\wh s+\pa_I\Phi^\pm (I+\eps
(M_1^I\circ\SSS^\pm-M_1^I),\wh x\wh y+\eps(M_1^r\circ\SSS^\pm-M_1^r))\\
&+\eps
(M_1^s\circ\SSS^\pm-M_1^s)+\OO^*(\eps^2\beta^{-4})\\
 \wh y^\pm=&\wh x^{-1} e^{-\pa_r\Phi^\pm (\wh I,\wh x\wh
y)}+\OO^*(\eps\beta\ii)\\
\wh x^\pm=&e^{\pa_r\Phi^\pm
(\wh I,\wh x\wh y)}\wh x\left(\wh
x\wh
y+\eps(M_1^r\circ\SSS^\pm-M_1^r)\right)+\OO^*(\rr^2+\rr\eps\beta\ii+\eps^2\beta^
{ -3 } )
\end{split}
\]
Moreover,
\[
\wh y^\pm \wh x^\pm=\wh y \wh x +\eps
(M_1^r\circ\SSS^\pm-M_1^r)+\eps^2(M_2^r\circ\SSS^\pm-\wt
M_2^r)+\OO^*(\eps^3\beta^{-5}). 
\]

\subsection{Gluing maps in resonances}\label{sec:gluingmaps:DR}
We consider the slow-fast variables $(\theta,J,x,y)$.
In these variables, 
we have the gluing map
\[
 (J^\pm, \theta^\pm,y^\pm,x^\pm)=\SSS^\pm(J,\theta,x,y)
\]
with
\[
\begin{split}
 J^\pm&=J\\
 \theta^\pm&=\theta+\pa_J\wt \Phi^\pm (J,xy)\\
 y^\pm&=x^{-1} e^{-\pa_r\wt \Phi^\pm (J,xy)}\\
 x^\pm&=x^{2}y e^{\pa_r\wt \Phi^\pm (J,xy)}\\
\end{split}
\]
where
\[
 \wt\Phi^\pm(I,r)=\Phi^\pm (k_0I,r).
\]
From now on, we drop the tilde in $\Phi^\pm$ to simplify the notation. 
We also abuse notation and we consider the normal form change of 
coordinates given by the generating function $W$  expressed in slow-fast 
variables (recall that all these changes of coordinates are symplectic).

Now we  express the gluing map in the normal form (slow-fast)
coordinates. As before, for the first four coordinates we have
\[
\begin{split}
 \wh J\pm =&\wh J+\eps (M_1^J\circ\SSS^\pm-M_1^J)+\OO^*(\eps^2) \\
 \wh \theta^\pm =&\wh \theta+\pa_J\Phi^\pm (J+\eps
(M_1^J\circ\SSS^\pm-M_1^J),\wh x\,\wh y+\eps (M_1^r\circ\SSS^\pm-M_1^r))\\
&+\eps
(M_1^\theta\circ\SSS^\pm-M_1^\theta)+\OO^*(\eps^2)\\
 \wh  y^\pm=&\wh x^{-1} e^{-\pa_r\Phi^\pm (J,\wh x\,\wh
y)}+\OO^*(\eps)\\
\wh  x^\pm=&e^{\pa_r\Phi^\pm
(J,\wh x\wh y)}\wh x\left(\wh
x\wh
y+\eps(M_1^r\circ\SSS^\pm-M_1^r)\right)+\OO^*(\eps^2)
\end{split}
\]
and
\[
\wh y^\pm \wh x^\pm=\wh y\, \wh x +\eps
(M_1^r\circ\SSS^\pm-M_1^r)+\eps^2(M_2^r\circ\SSS^\pm-\wt
M_2^r)+\OO_3^*(\eps).
\]

\section{The separatrix map in the non-resonant
regime}\label{sec:flowboxcoordinates:SR}

We use the results in Sections \ref{sec:flowestimates:SR} and 
\ref{sec:gluingmaps:SR} to look for formulas of the separatrix 
map in the  non-resonant regime \eqref{def:SR}.
First we compose the flow in normal form coordinates and 
the gluing map. We obtain the separatrix map in 
the $(I,s,y,x)$ coordinates. Later we  look for a good system 
of coordinates which will transform $(y,x)$ to a certain
symplectic flow-box coordinates around the former separatrix.

Consider the flow $g_\eps^{\ol t}$ analyzed in 
Lemma \ref{lemma:Flow:SR} and the gluing 
map analyzed in Section \ref{sec:gluingmaps:SR}.
We have the following 

\begin{lemma}
The composition of the two maps 
$\FF=g_\eps^{\ol t}\circ G_\eps^\sigma$ is given
by
\[\FF(I,s,y,x)=(\FF_I(I,s,y,x),\FF_s(I,s,y,x),\FF_y(I,s,y,x),\FF_x(I,s,y,x))\]
with
\[
\begin{split}
\FF_I= &  I+\eps
\PP_1^I+\eps^2\PP_2^I+\OO^*(\eps^3\beta^{-4}+\eps^2\beta^{-2}
\rr+\eps\rr^2)|\log\rr| \\
\FF_s=&s+\pa_I\Phi^\pm (I,xy)+(\nu(I)+\pa_I g(I,xy))\ol t+\eps
\PP_1^s+\OO^*(\eps^2\beta^{-4})|\log\rr|\\
\FF_y=&x\ii \exp(-\pa_r g(I+\eps \PP_1^I,xy+\eps\PP_1^r) \ol t-\pa_r
\Phi^\pm(I,xy))\times \\
& \left(1+\OO_2^*(\eps,\rr)|\log\rr|\right)+\OO_3^*(\eps,\rr)\\
\FF_x=&\exp(\pa_r g(I+\eps \PP_1^I, xy+\eps\PP_1^r) \ol t
+\pa_r \Phi^\pm(I,xy))) \times \\
& x(xy+\eps\PP_1^r)\left(1+\OO_2(\eps,\rr)|\log\rr|\right),
\end{split}
\]
where
\[
\begin{split}
 \PP^I_1=&M_1^I\circ\SSS^\pm-M_1^I\\
 \PP^I_2=&M_2^I\circ\SSS^\pm-\wt M_2^I
 \end{split}
\]
and
\[
 \begin{split}
\PP^s_1&=M_1^s\circ\SSS^\pm-M_1^s+\pa_I^2\Phi^\pm(I,
xy)\left(M_1^I\circ\SSS^\pm-M_1^I\right)+\\
&+\pa^2_{r\,I}\Phi^\pm(I,
xy)\left(M_1^r\circ\SSS^\pm-M_1^r\right)
+\ol t\pa_I(\nu(I)+\pa_I g(I,xy))\left(M_1^I\circ\SSS^\pm-M_1^I\right)\\ 
&+\ol t\,
\pa^2_{r\,I} g(I,xy) \left(M_1^r\circ\SSS^\pm-M_1^r\right)
 \end{split}
\]
Moreover,
\[
 \FF_x\FF_y=xy+ \eps \PP^r_1+\eps^2
\PP^r_2+\OO^*_3\left(\eps\beta\ii,
\rr\right),
\]
where
\[
\begin{split}
 \PP^r_1=&M_1^r\circ\SSS-M_1^r\\
 \PP^r_2=&M_2^r\circ\SSS-\wt M_2^r
 \end{split}
\]
Then,
\[
 \PP_1^I,\PP_1^r=\OO^*(\beta\ii),\quad 
\PP^I_2,\PP_2^r=\OO^*(\beta^{-3}),
\quad\text{and}\quad\PP_1^s=\OO^*(1)|\log\rr|.
\]
\end{lemma}
The proof of this lemma is a direct consequence of the results in Lemma
\ref{lemma:Flow:SR} and Section \ref{sec:gluingmaps:SR}.

We look for a coordinate change near the former 
separatrix such that formulas for the separatrix map are 
as simple as possible. In comparison with \cite{Treschev02a}, 
we want to point out two main differences. First, since we are 
away from resonances, we have ${\bf \ol H}_j=0,\ j=1,2,3$. 
This simplifies the formulas. On the other hand, we
want to have a more precise dependence on $xy$ since we are doing a higher 
order analysis. This second fact implies that we have to slightly
modify the change of coordinates.

We look for a symplectic change
\[
\left(\begin{split} s\\I\\ y\\ x\end{split}\right)=\Upsilon \left(\begin{split} 
\xi\\\eta\\ \tau\\ h\end{split}\right) 
\]
The function $I\mapsto g(I,r)$ introduced in Lemma \ref{lemma:Moser}
is invertible with respect to $r$ in a neighborhood of $r=0$ 
(recall that $\pa_r g(I,0)=\la(I)>0$, see Lemma \ref{lemma:Moser}). 
We have denoted by  $g_r^{-1}(I,r)$
 the inverse function with respect to the second coordinate $r$. 
 We consider the following generating function. 
\begin{equation}\label{def:GeneratingFunctionFlowBox:SR}
\SSS(\eta,s,h,y)=\eta s+g_r^{-1}(\eta,h-E(\eta))\log|y|.
\end{equation}
This generating function induces the following change of coordinates
\[
 \begin{pmatrix} \xi\\\eta\\ \tau\\ h\end{pmatrix}=
\begin{pmatrix} s-\pa_r g_r\ii (I,g(I,xy))(\nu(I)+\pa_I g(I,xy))\log|y|\\I\\
\frac{\log |y|}{\pa_r g(I,xy)}\\ E(I)+g(I,xy)\end{pmatrix}
\]
and 
\begin{equation}\label{def:FlowBoxChange}
\begin{pmatrix} s\\I\\ y\\ x\end{pmatrix}= \Upsilon \begin{pmatrix} \xi\\\eta\\
\tau\\ h\end{pmatrix}=\begin{pmatrix}
\xi+[\nu(\eta)+\pa_\eta g\ii (\eta,g_r\ii (h-E(\eta)))]\tau\\ \eta 
\\ \sigma\,
\exp(\frac{\tau}{\pa_r g_r\ii (\eta,h-E(\eta))})\\ 
g_r\ii (\eta,h-E(\eta))\,\sigma\,
\exp(-\frac{\tau}{\pa_r g_r\ii (\eta,h-E(\eta))})\end{pmatrix}
\end{equation}
Note that the $y$ component does not get modified by 
this change of coordinates and that this change of 
coordinates does not depend on $\eps$.

We express the separatrix map in these coordinates. We use the change
$\Upsilon$ to write down the formulas. From now on we omit the dependence on
$\beta$. Recall that $\beta>0$ is a fixed parameter independent of $\eps$. Due
to the previous analysis, 
smoothness of the separatrix map obeys the estimate in Theorem 
\ref{thm:SM:SR}.

\begin{lemma}\label{lemma:SM:SR:1}
 The separatrix map, has the following form
 \[
  \begin{split}
 \eta^*=&\eta+\eps\PP_1^I\circ \Upsilon+\eps^2\PP_2^I\circ
\Upsilon+\OO_3^*(\eps,\rr)|\log\rr|\\
   \xi^*=&\xi+\pa_I\Phi^\pm(\eta,g_r\ii(\eta,
h-E(\eta)))-\frac{\nu(\eta)+\pa_I g(\eta,g_r\ii(\eta,
h-E(\eta)))}{\pa_r g(\eta,g_r\ii (\eta,h-E(\eta))}\log
R\\
&+\OO_1^*(\eps+\rr)(|\log\eps|+|\log\rr|)\\
 h^*=&h+\eps\left(\left[\nu(\eta)+\pa_I g(\eta,g_r\ii(\eta,
h-E(\eta)))\right]\PP_1^I\circ \Upsilon \right. \\
& \left.+\pa_r g(\eta,g_r^{-1}(\eta,h-E(\eta)))\PP_1^r\circ \Upsilon\right)
+\eps^2\PP_2^h+\OO_3^*(\eps,\rr)\\
 \tau^*=&\tau+\ol t+\frac{1}{\pa_r g(\eta,g_r\ii (\eta,h-E(\eta))}\log
R+\OO_1^*(\eps+\rr)(|\log\eps|+|\log\rr|),
  \end{split}
 \]
 where
\begin{equation}\label{def:R}
R= e^{\pa_r\Phi^\pm(\eta,g_r\ii
(\eta,h-E(\eta))}\left(g_r\ii(\eta^*,h^*-E(\eta^*))+\OO_2^*\right)
\end{equation}
and 
\[
 \begin{split}
\PP_2^h&=  \left(\nu(\eta)+\pa_Ig(\eta,g_r\ii(\eta,
h-E(\eta)))\right)
\PP_2^y\circ \Upsilon\\ 
& +\pa_r
g(\eta,g_r\ii(\eta,
h-E(\eta)))\PP_2^r\circ \Upsilon\\
&+\frac{1}{2} \left(\pa_I\nu(\eta)+\pa_I^2 g(\eta,g_r\ii(\eta,
h-E(\eta)))\right)
\left(\PP_1^I\circ \Upsilon\right)^2\\
&+\frac{1}{2} \pa^2_r
g(\eta,g_r\ii(\eta, h-E(\eta)))\left(\PP_1^r\circ \Upsilon\right)^2.
 \end{split}
\]
\end{lemma}
\begin{proof}
We start with the  $\eta$ component. We have that
\[
 \eta^*=y^*=y+\eps \PP_1^I+\eps^2\PP_2^I+\OO_3(\eps,\rr)|\log\rr| 
\]
Then, it is enough to apply the change $\Upsilon$ defined in
\eqref{def:FlowBoxChange} to obtain the formula for $\eta^*$.

For the $h$ component we have,
\[
\begin{split}
 h^*&=E(I^*)+g(I^*,x^*y^*)\\
 &=E(I)+g(I,xy)\\
 &+\left(\nu(I)+\pa_Ig(I,xy)\right)
\left(\eps\PP_1^I+\eps^2\PP_2^I\right)+\pa_r
g(I,xy)\left(\eps\PP_1^r+\eps^2\PP_2^r\right)\\
&+\frac{\eps^2}{2} \left(\pa_I\nu(I)+\pa_I^2 g(I,xy)\right)
\left(\PP_1^I\right)^2+\frac{\eps^2}{2} \pa^2_r
g(I,xy)\left(\PP_1^r\right)^2+\OO_3(\eps,\rr)|\log\rr| 
\end{split}
\]
Apply the change $\Upsilon$, one obtains the formula for $h^*$.

To compute the $\tau$ component we use the following identity,
\[
 g_r\ii (\eta^*, h^*-E(\eta^*))=g_r\ii (\eta, h-E(\eta))+\eps\PP_1^r+\OO_2^*.
\]
We also have
\[
 \begin{split}
 \tau^*=&\frac{\log|y^*|}{\pa_r g(I^*,x^*y^*)}\\
 =&\frac{\log|y|+\pa_r g(I+\eps \PP_1^I,xy+\eps\PP_1^r)\ol
t+\log\left(e^{\pa_r\Phi^\pm (I,xy)}\left(xy+\eps\PP_1^r\right)\right)}{\pa_r 
g(I+\eps
\PP_1^I,xy+\PP_1^r)}+\OO_2|\log\rr|.
 \end{split}
\]
Then, we obtain
\[
 \tau^*=\tau+\ol t+\frac{1}{\pa_r g(\eta,g_r\ii (\eta,h-E(\eta))}\log
R+\OO_1^*(\eps+\rr)(|\log\eps|+|\log\rr|),
\]
where $R$ is the function introduced in \eqref{def:R}.

Proceeding analogously, one can  compute the $\xi$ component,
\[
\begin{split}
 \xi^*=&\xi+\pa_I\Phi^\pm(\eta,g_r\ii (\eta,h-E(\eta)))-\frac{\nu(\eta)+\pa_I
g(\eta,g_r\ii(\eta,
h-E(\eta)))}{\pa_r g(\eta,g_r\ii (\eta,h-E(\eta))}\log
R\\
&+\OO_1^*(\eps+\rr)(|\log\eps|+|\log\rr|).
\end{split}
\]
Note that $\pa_I g(\eta,g\ii(\eta,
h-E(\eta)))$ satisfes 
\[
 \pa_I g(\eta,g_r\ii(\eta,
h-E(\eta)))=\OO_1^*(\eps,\rr).
\]
Therefore, 
\[
\begin{split}
 \xi^*=&\xi+\pa_I\Phi^\pm(\eta,g_r\ii (\eta,h-E(\eta)))-\frac{\nu(\eta)}{\pa_r
g(\eta,g_r\ii
(\eta,h-E(\eta))}\log
R\\
&+\OO_1^*(\eps+\rr)(|\log\eps|+|\log\rr|).
\end{split}
\]
This completes the derivation of the separatrix map in the non-resonant case.
\end{proof}

Theorem \ref{thm:SM:SR} is a direct consequence of this lemma. We define  the 
first orders in the action components in Theorem \ref{thm:SM:SR} as
\[
\begin{split}
 M_1^{\sigma,\eta}&=\PP_1^I\circ\Upsilon\\
M_1^{\sigma,h}&=\left[\nu(\eta)+\pa_I g(\eta,g_r\ii(\eta,
h-E(\eta)))\right]\PP_1^I\circ \Upsilon
+\pa_r g(\eta,0)\PP_1^r\circ \Upsilon.
\end{split} 
 \]
The second orders can be also easily defined.
 
We finish this 
section by identifying the order $\eps$ of the separatrix map in terms of the 
Melnikov potential. This allows to compare our results with the results in 
\cite{Treschev02a}.

\begin{lemma}\label{lemma:FirstOrderMs}
 The functions  $M_1^{\sigma,\eta}$ and  $M_1^{\sigma,h}$ satisfy
 \[
   M_1^{\sigma,\eta}=\pa_\xi\Theta^\sigma +\OO(h-E(\eta)),\qquad 
M_1^{\sigma,h}=\pa_\tau\Theta^\sigma +\OO(h-E(\eta))
 \]
where $\Theta^\sigma$ are the splitting potentials defined in 
\eqref{def:Melnikov}.
\end{lemma}

\begin{proof}
We explain how to prove the statement for the $\eta$ component. The other one 
can be obtained analogously. 

In Lemma \ref{lemma:Flow:SR} we have seen that 
\[
M_1^{\sigma,1}
=\PP_1^I\circ\Upsilon=M_1^I\circ\SSS\circ\Upsilon-M_1^I\circ\Upsilon.
\]
We start by analyzing the function $M_1^I=\pa_sW_0$ given in Lemma 
\ref{lemma:NormalFormChangeOfCoordinates}.

By the definition of the function $W_0$ in Lemma
\ref{lemma:NormalForm:FirstOrderCohomological}, we have that 
$W_0=W_{00}+\OO(xy)$ where $W_{00}$ is defined as follows. Recall that 
$\pa_I g(I,r)=\OO(r)$  and $\pa_r g(I,r)=\la(I)+\OO(r)$ with $\la(I)>0$. The 
function $W_{00}$ can be split as 
\[
W_{00}=W_{00}^ 1+W_{00}^ 2+W_{00}^3,
\] 
with
\[
 \begin{split}
  W_{00}^1&=\pa_0\ii \HH_1^{(1)}\\
   W_{00}^2&=-\int_{-\infty}^0 \HH_1^{(2)}(I,s+\nu(I)t', ye^{\la(I) t'}, 
t+t')\,dt'\\
W_{00}^3&=-\int^{+\infty}_0 \HH_1^{(3)}(I,s+\nu(I)t', xe^{-\la(I) t'}, 
t+t')\,dt'\\
 \end{split}
\]
where $\HH^{j}_1$ are the Hamiltonians introduced in 
\eqref{def:HamiltonianSplitting} and 
\[
 \pa\ii (f)=\pa\ii \left(\sum_{(k,k_0)\in\ZZ^{n+1}}f^{k,k_0}(I)e^{2\pi 
i (ks+k_0 t)}\right)=\qquad \qquad \qquad \]
\[\sum_{(k,k_0)\in\ZZ^{n+1}}\frac{f^{k,k_0}(I)}{2\pi i 
\left(k\nu(I)+k_0\right)}e^{2\pi 
i (ks+k_0 t)},
\]
which is well defined in the non-resonant zone. These functions are the first 
order in $xy$ of the funtions $W_0^j$, $j=1,2,3$, introduced in the proof of 
Lemma \ref{lemma:NormalForm:FirstOrderCohomological}.


Therefore, $M_1^I=\pa_s W_{00}^1+\pa_s W_{00}^2+\pa_s W_{00}^3+\OO(xy)$. Now, 
in Section \ref{sec:gluing:H0} we have seen that $y^\pm ,x=\OO(xy)$. 
Moreover, from the definition of the Hamiltonians $\HH^{j}_1$  in 
\eqref{def:HamiltonianSplitting} one has that $\HH^{2}_1=\OO(x)$ and 
$\HH^{3}_1=\OO(y)$. Thus, 
\[
 M_1^I\circ\SSS-M_1^I= 
\pa_s W_{00}^1\circ\SSS
-\pa_s W_{00}^1-\pa_s W_{00}^2
+\pa_s W_{00}^3\circ\SSS+\OO(xy).
 \]
Moreover, as shown in \cite{Treschev02a} and recalled in Section
\ref{sec:gluing:H0}, the functions $\Phi^\pm$ satisfy
$\pa_I\Phi^\pm(I,r)=\mu^\pm(I)+\OO(r)$ and 
$e^{\pa_r\Phi^\pm (I,r)}=\kk^\pm(I)+\OO(r)$. Then, we obtain
\[
\begin{split}
 M_1^I\circ\SSS-M_1^I= \,&
\pa_s W_{00}^1(I,s+\mu^\sigma(I),t)-\pa_s W_{00}^1(I,s,t)\\
&+\pa_s W_{00}^3(I,s+\mu^\sigma(I),
y,t)-\pa_s W_ {00}^2(I,s,y,t)+\OO(xy)\\
= \,&
\pa_s W_{00}^1(I,s+\mu^\sigma(I),t)-\pa_s W_{00}^1(I,s,t)\\
&+\pa_s W_{00}^3(I,s+\mu^\sigma(I),
(\kk^\sigma(I)y)\ii,t)-\pa_s W_ {00}^2(I,s,y,t)+\OO(xy).
\end{split}
 \]
Now it only remains to apply the change of coordinates $\Upsilon$. The
change 
$\Upsilon$ satisfies
\[
 \Upsilon(\xi,\eta,h,\tau)=(\xi+\nu(\eta)\tau, \eta, \sigma e^{\la(\eta)\tau}, 
0)+\OO(h-E(\eta)).
\]
We apply this change of coordinates to each term. We obtain first 
\[
\begin{split}
 \left.\pa_s W_{00}^1(I,s+\mu^\sigma(I),0)-\pa_s 
W_{00}^1(I,s,0)\right|_{(I,s,0)=\Upsilon(\eta,\xi,\tau)}=\\
\quad\quad\vartheta(\eta,
\xi+\nu(\eta)\tau , 0)-\vartheta(\eta , 
\xi+\nu(\eta)\tau+\mu^\sigma(\eta),0),
\end{split}
\]
where $\vartheta$ is the function introduced in \eqref{def:InnerAveraging}.

For the two other terms, it is enough to use to see 
\[
 \begin{split}
 \left.\pa_s W_ {00}^2(I,s,y,0)\right|_{(I,s,y,0)=\Upsilon(\eta,\xi,\tau)}&= \\
\pa_s W_ 
{00}^2(\eta,\xi+\nu&(\eta)\tau,\sigma 
e^{\la(\eta)\tau},0)\\
\left.\pa_s W_ 
{00}^3(I,s,(\kk(I)y)\ii,0)\right|_{(I,s,y,0)=\Upsilon(\eta,\xi,\tau)}&=\\
  \pa_s 
W_{00}^3(\eta,\xi+&\nu(\eta)\tau+\mu^\sigma(\eta),
\kk(\eta)\ii\sigma 
e^{-\la(\eta)\tau},0).
 \end{split}
\]

\

\begin{lemma}The following identity is satisfied,
\[
\begin{split}
\vartheta(\eta,
\xi+\nu(\eta)\tau , 0)-\vartheta(\eta , 
\xi+\nu(\eta)\tau+\mu^\sigma(\eta)\tau,0)&\\
-\pa_s W_ 
{00}^2(\eta,\xi+\nu(\eta)\tau,\sigma 
e^{\la(\eta)\tau},0)&\\
+ \pa_s 
W_{00}^3(y,\xi+\nu(\eta)\tau+\mu^\sigma(\eta),
\kk(\eta)\ii\sigma 
e^{-\la(\eta)\tau},0)&=\\
\vartheta(\eta,
\xi, -\tau)&-\vartheta(\eta , 
\xi+\mu^\sigma(\eta),\tau)\\
+\pa_s W_ 
{00}^2(\eta,\xi,\sigma 
,-\tau)&+ \pa_s 
W_{00}^3(y,\xi+\mu^\sigma(\eta),
\kk(\eta)\ii\sigma ,-\tau).
\end{split}
\]
\end{lemma}

\begin{proof} This lemma follows from the 
definition (\ref{def:InnerAveraging}). One can 
expand both sides into Fourier series and match them. 
\end{proof}

From this lemma, one can easily deduce the statement of Lemma
\ref{lemma:FirstOrderMs}.
\end{proof}

\section{The separatrix map in the neighborhood
of resonances}\label{sec:flowboxcoordinates:DR}
In the resonant regime we only compute the system up 
to first order. Thus, we follow closely \cite{Treschev02a}. 
The main difference is that our  resonant region is 
much larger than in \cite{Treschev02a}. Indeed, our 
$\beta$ is fixed indepedent of $\eps$, whereas in \cite{Treschev02a} he 
considers 
$\beta=\eps^{1/4}$. 
Recall that we are not keeping track of the dependence 
on $\beta$ and that we have assumed $\rr\lesssim\eps$.

First, we compose the flow in the normal form coordinates and 
the gluing map. We obtain the separatrix map in the 
$(J,\theta,x,y,t)$ coordinates. 

We denote the composition of the two maps $\FF=g_\eps^{\ol t}\circ
G_\eps^\sigma$ by
\[\FF(J,\theta,x,y)=(\FF_J(J,\theta,x,y),\FF_\theta(J,\theta,x,y),\FF_x(J,
\theta ,x,y) , \FF_y(J,\theta ,x, y)).\]
Recall that the map $\FF$  is independent of $t$ because we choose the initial 
time as $t=0$. 

For the $J$ component
one can easily see that 
\[
 \FF_J=J+\eps \PP^J_1+\OO^*(\eps^{11/6}),
 \]
where
\[
\PP^J_1(J,\theta,x,y)=G_1(J,\theta+\mu^\sigma(J),x,y)+M_1^J\circ\SSS(J,\theta,x
,y)-M_1^J(J,\theta,x,y)
 %
\]
where the function $G_1$ has been introduced in Section \ref{sec:gluingmaps:DR}
and the function $M^I$ has been used to define the gluing maps in Section
\ref{sec:gluingmaps:DR}.
From Lemma \ref{lemma:NormalFormChangeOfCoordinates}, Lemma \ref{lemma:Flow:DR}
and the definition of the gluing map in Section \ref{sec:gluingmaps:DR}, we can 
deduce that
 $\PP^J_1
 =\OO^*(1)$.

For the $\theta$ variable we have 
\[
\begin{split}
 \FF_\theta=&\theta+\pa_J\Phi^\pm (J,xy)+ (\nu(J)+\pa_J g(J,xy))\ol t\\
&+\eps
\PP^\theta_1+\OO^*(\eps^{11/6})
\end{split}
\]
where
\[
\begin{split}
\PP^\theta_1=&F_1+ M_1^\theta\circ\SSS-M_1^\theta+\pa_J^2\Phi^\pm(J,
xy)(M_1^J\circ\SSS-M_1^J)\\
+&\ol t\left(\pa_J\nu(J)+\pa^2_{J} g(J,xy))\right)
(M_1^J\circ\SSS-M_1^I)+\ol
t\pa^2_{Jr} g(J,xy))(M_1^r\circ\SSS-M_1^r).
\end{split}
\]
Using Lemmas \ref{lemma:NormalFormChangeOfCoordinates} and \ref{lemma:Flow:DR},
one can check that $\PP^\theta_1=\OO^*(\log\eps)$.

For the $x$ and $y$ components we have,
\[
\begin{split}
\FF_x&=y\ii \exp(
-\left(\pa_r g(J+\eps \PP_1^J, xy+\eps\PP_1^r) \ol
t+\Phi(J,\theta,\ol t)+\pa_r
\Phi^\pm(J,xy) \right))\times \\ &
\left(1+\OO_2^*(\eps)|\log\eps|\right)+\OO_3^*(\eps)\\
\FF_y&=\exp(\left(\pa_r g(J+\eps \PP_1^J, xy+\eps\PP_1^r) \ol
t+\Phi(I,\theta,\ol t)+\pa_r
\Phi^\pm(J,xy)\right) \times \\
& y(xy+\eps\PP_1^r)\left(1+\OO_2(\eps)|\log\eps|\right)+\OO_3^*(\eps).
\end{split}
\]
Moreover,
\[
 \FF_x\FF_y=xy+\eps\PP_1^r+\OO(\eps^2)
\]
where
\[
 \PP_1^r=M_1^r\circ\SSS-M_1^r,
\]
by Lemma \ref{lemma:NormalFormChangeOfCoordinates}, satisfies $
\PP_1^r=\OO(\beta^{-1})$.

To have simpler formulas for the separatrix map we consider a  change of
coordinates slightly different from the ones used in 
\cite{Treschev02a} and  in Section \ref{sec:flowboxcoordinates:SR}. The reasons 
are the following. On the one hand, we have to deal with a 
rather large resonant region   since it has width of order one with 
respect to $\eps$ instead of order $\eps^{1/4}$ as in \cite{Treschev02a}. On 
the other hand, we do not need to consider higher orders as  in Section 
\ref{sec:flowboxcoordinates:SR} since  in
the resonant regime we just analyze the first order of the separatrix map. As
in the single resonant regime, we  assume that  $|xy|\leq\eps$. We consider the
change defined by the generating function
\begin{equation}\label{def:GeneratingFunctionFlowBox:DR}
\SSS(J, \xi,x,\tau)=J\xi +E(J)\tau+\sigma x e^{\la(J)\tau}+\eps\int_0^\tau
\overline{\bf H}_1
(J,\xi+\nu(J)s)ds.
\end{equation}
Note that in the core of the resonance $\nu(J)\sim 0$ and, therefore,
\[
 \int_0^\tau
\overline{\bf H}_1
(J,\xi+\nu(J)s)ds\ \sim\ \tau 
\overline{\bf H}_1
(J,\xi)\ \sim\  \tau 
\overline{\bf H}_1
(J,\xi+\nu(J)\tau)
\]
as in \cite{Treschev02a}. Nevertheless, the change we consider is better suited 
for points not extremely close to the double resonance. 

Recall that after switching to slow-fast variables,  
$\overline{\bf H}_1$ does not
depend on $t$.
We have then the following changes 
\[
 \begin{split} 
 \xi=& \theta-\frac{\nu(J)}{\la(J)}\log 
|y|+\la\ii(J)\la'(J)xy+\eps\int_0^\tau{\bf\bar
H}_1(J,\theta+\nu(J)(s-\la\ii(J)\log|y|))ds\\
&+\eps\tau\nu\ii(J)\nu'(J) \left({\bf\bar
H}_1(J,\theta)-{\bf\bar
H}_1(J,\theta-\la\ii(J)\log|y|)\right)+\OO^*(\eps,xy)\\
\eta=&J+\eps\nu\ii(J)\left({\bf\bar
H}_1(J,\theta)-{\bf\bar
H}_1(J,\theta-\nu(J)\la\ii(J)\log|y|)\right)+\OO_2^*(\eps,xy)\\
\tau=&\frac{\log|y|}{\la(J)}+\OO^*(\eps,xy)\\ 
h=&E(J)+\la(J)xy+\eps {\bf\bar
H}_1(J,\theta)+\OO_2^*(\eps,xy) \end{split}
\]
and 

\begin{equation}\label{def:FlowBoxChange:DR}
\begin{pmatrix} \theta\\J\\ x\\ y\end{pmatrix}= \Upsilon \begin{pmatrix}
\xi\\\eta\\
\tau\\ h\end{pmatrix}
\end{equation}
defined as
\begin{equation} \nonumber 
\begin{split}
\theta=&\xi+\nu(\eta)\tau+\eps\int_0^\tau\pa_J {\bf\bar
H}_1(\eta,\xi+\nu(\eta)s)ds\\
&+\la'(\eta)\la\ii(\eta)\tau \left(h-E(\eta)- {\bf\bar
H}_1(\eta,\xi)\right)+\eps\OO^*(\eps , h-E)  \\ 
J=&\eta-\eps\nu\ii(\eta)\left({\bf\bar
H}_1(\eta,\xi+\nu(\eta)\tau)-{\bf\bar
H}_1(\eta,\xi)\right)+\eps\OO^*(\eps, h-E) \\
x=&\frac{\sigma}{\la(\eta)}
e^{-\la(\eta)\tau}\left(h-E(\eta)-\eps 
{\bf\bar
H}_1(\eta,\xi)\right)+\eps\OO^*(\eps, h-E) \\ 
y=&\sigma e^{\la(\eta)\tau}+\OO^*(\eps, h-E)
\end{split}
\end{equation}
We call $\Upsilon_0$ the first order of this change of coordinates, defined by 
\begin{equation}\label{def:FlowBoxChange:DR:0}
\begin{pmatrix} \theta\\J\\ x\\ y\end{pmatrix}= \Upsilon_0 \begin{pmatrix}
\xi\\\eta\\
\tau\\ h\end{pmatrix}=\begin{pmatrix}
\xi+\nu(\eta)\tau \\ 
\eta\\
\frac{\sigma}{\la(\eta)}
e^{-\la(\eta)\tau}\left(h-E(\eta)\right) \\ \sigma
e^{\la(\eta)\tau}\end{pmatrix}
\end{equation}
Theorem \ref{thm:SM:DR} can be rephrased as 
the following lemma. 

\begin{lemma}\label{lemma:SM:DR:1}
Assume that the function $w_0$ in \eqref{def:omega01} satisfies
\[
 c\ii\eps^{1+a}<|w_0(\eta^*,h^*,\xi^*)|<c\eps
\]
for some $c>0$ and $1\ge a>0$ independent of $\eps$. Then, 
the separatrix map $\SM_\eps$ has the following form
\begin{equation*}
\begin{aligned}
\eta^*=&\eta&&+\eps\pa_\xi\Theta^\sigma(\eta,\xi,\tau)&&+\eps 
B^{\eta,\sigma} (\eta,\xi,w_0,\tau)&&+\OO^*(\eps^{5/3})\\
 \xi^*=&\xi+\mu^\sigma&&+\frac{ 
\nu}{\la}\log\left|\frac{\kk^\sigma 
w_0^\sigma}{\la}\right|&&+\eps B^{\xi,\sigma} (\eta,\xi,w_0,\tau, \ol 
t)&&+\OO^*(\eps)\\
 h^*=&h &&+\eps\pa_\tau\Theta^\sigma(\eta,\xi,\tau)&&+\eps 
B^{h,\sigma} (\eta,\xi,\tau)&& +\OO^*(\eps^{5/3})\\
 \tau^*=&\tau+\ol t&&+ \frac{1}{\la}\log\left|\frac{\kk^\sigma 
w_0^\sigma}{\la}\right|&&+\eps B^{\tau,\sigma}(\eta,\xi,w_0,\tau, \ol 
t)&&+\OO^*(\eps)\\
\sigma^*=&\sigma\ \mathrm{sgn}\ w_0
\end{aligned}
\end{equation*}
where $\nu$, $\la$, $\mu^\sigma$ and $\kk^\sigma$ are functions of $\eta$,
\[
 \begin{split}
B^{\eta,\sigma} (\eta,\xi,w_0^\sigma,&\tau)= 
 -\frac{1}{\nu}\left({\bf\bar
H}_1(\eta,\xi+\nu\tau)-{\bf\bar
H}_1(\eta,\xi)\right)\\
&\frac{1}{\nu}\left({\bf\bar
H}_1(\eta,\xi+\nu\tau+\mu^\sigma)-{\bf\bar
H}_1(\eta,\xi+\frac{\nu}{\la}\log\left|\frac{\kk^\sigma 
w_0}{\la}\right|+\mu^\sigma)\right)\\
B^{h,\sigma} (\eta,\xi,&\tau)= 
\ {\bf\bar
H}_1(\eta,\xi+\nu\tau+\mu^\sigma)-{\bf\bar
H}_1(\eta,\xi+\nu\tau),
\end{split}
\]
and
\[
 \begin{split}
 B^{\xi,\sigma} (\eta,\xi,w_0,\tau,\ol t)&= 
f_1(\eta,\xi,w_0,\tau)\tau+f_2(\eta,\xi,w_0,\tau)\log \left|\frac{\kk^\sigma 
w_0}{\la}\right|+f_3(\eta,\xi,w_0,\tau, \ol t)\ \ol t\\
 B^{\tau,\sigma} (\eta,\xi,w_0,\tau,\ol t)&=  
g_1(\eta,\xi,w_0,\tau)\tau+g_2(\eta,\xi,w_0,\tau)\log \left|\frac{\kk^\sigma 
w_0}{\la}\right|+g_3(\eta,\xi,w_0,\tau, \ol t)\ \ol t
\end{split}
\]
for certain functions $f_i$, $g_i$ which satisfy $f_i,g_i=\OO^*(1)$. Therefore, 
the functions $B^{z,\sigma}$ satisfy
\[
 B^{\eta,\sigma}, B^{h,\sigma}=\OO^*(1),\qquad B^{\xi,\sigma}, 
B^{\tau,\sigma}=\OO^*(\log\eps).
\]
\end{lemma}
Recall that $\nu(0)=0$. Nevertheless, one can easily see that the function $ 
B^{\eta,\sigma}$ is well defined even as $\nu\rightarrow 0$ since it has a well 
defined limit.

\begin{proof}
To prove this lemma, it is enough to compose the change of coordinates given in 
\eqref{def:FlowBoxChange:DR} with the map $\FF$ as done in the proof of Lemma 
\ref{lemma:SM:SR:1}. One has also to take into account how to 
derive the splitting potentials $\Theta^\sigma$ from the definitions of 
the functions $\PP_1^z$, as explained in Lemma \ref{lemma:FirstOrderMs}.

\end{proof}

\appendix
\section{The separatrix map of the generalized Arnold
example}\label{sec:Arnold}

In this appendix we apply Theorems
\ref{thm:SM:SR} and \ref{thm:SM:DR}  to the generalized 
Arnold example \eqref{def:ArnoldGeneralized}. The Arnold example 
presents several simplifications and thus the
formulas are considerably simpler than in the general case. 
For these models all the transformations and maps 
are $\CCC^{\infty}$. The results presented in this appendix also apply to 
perturbations of the form 
\[
H_1(I,\varphi,p,q,t):=
P(I,p,\exp(iq),\exp(i \varphi),\,\exp( it)), 
\]
where $P$ is a real trigonometric polynomial of 
order $N$ in $\varphi$ and $t$, i.e.  it has the form: 
\[
\sum_{|k_i|\le N, i=1,2,3}  
h_{k_1,k_2,k_3} (p,q,I)\exp i (k_1q+k_2 \varphi+k_3 t),
\]
for some $N\in \Z_+$.

The key point is that the generalized Arnold example has an unperturbed 
Hamiltonian $H_0$ which does not present copuling between the ``pendulum 
variables'' and the rotator. We first analyze the non-resonant regime as defined 
in \eqref{def:SR}.  Thus, we consider the resonant
regime for any resonance $k\in\NNN(H_1)\cup \NNN^{(2)}(H_1)$ as defined 
in \eqref{def:DR}.

For the generalized Arnold example, the frequency $\nu(\eta)$ defined in 
\eqref{def:nu} satisfies $\nu(\eta)=\eta$. The matrix $\La$ defined in 
\eqref{def:MatrixLambda} is just
\[
 \La=\begin{pmatrix} 0&1\\1&0\end{pmatrix}
\]
therefore its positive eigenvalue $\la$ is $\la=1$ and is
independent of $I$. The function $g$, defined through the Moser normal form in 
Lemma
\ref{lemma:Moser}, is independent of $I$  and satisfies
$g(r)=r+\OO_2(r)$. Then, the function $w$  in \eqref{def:omega} is
defined by 
\[
w=g\ii \left(h^*-\frac{(\eta^*)^2}{2}\right).
\]
Analogously, since the Moser normal form is independent of $I$, so are the 
functions $\Phi^\pm$ involved in the gluing map. Thus, 
the functions $\mu^\pm$ in \eqref{def:mu} satisfy $\mu^\pm=0$. 

We have the following theorem.

\begin{theorem}\label{thm:SM:SR:Arnold}
Fix $\beta>0$ and $1\ge a>0$. For $\eps$ sufficiently small
there exist $c>0$ independent of $\eps$ and a canonical 
system of coordinates $(\eta,\xi,h,\tau)$ such that in 
the non-resonant zone $ \Sr_\beta$ we have 
\[
\eta = I +\OO^*_{1}\left(\eps,H_0-\frac{I^2}{2}\right), \quad 
\xi +\eta\tau=\varphi+f, \quad 
h=H_0+\OO^*_{1}\left(\eps,H_0-\frac{I^2}{2}\right),
\]
where $f$ denotes a function depending only on 
$(I,p,q,\eps)$ and such that 
$f=\mathcal O(\eps)$. In these coordinates 
the separatrix map has the following form. For any 
$\sigma\in \{-,+\}$ and $(\eta^*,h^*)$ such that 
\[
 c\ii\eps^{1+a}<|w(\eta^*,h^*)|<c\eps,\qquad |\tau|<c\ii, \qquad
c<|w(\eta^*,h^*)|\,e^{\ol t}<c\ii,
\]
the separatrix map $(\eta^*, \xi^*,h^*,\tau^*)=
\SM(\eta, \xi,h,\tau)$ is defined implicitly as follows
 \[
  \begin{split}
 \eta^*=&\ \eta- \ \ \eps
M_1^{\sigma,\eta}+\ \ \eps^2 M_2^{\sigma,\eta}+\ 
\ \OO_3^*(\eps\log\eps)\\
   \xi^*=&\ \xi+\qquad \qquad \ \ \ \ \ \pa_\eta w(\eta^*,h^*) 
\left[\log |w(\eta^*,h^*)| + (\Phi^\sg)' ( w (\eta^*,h^*))
\right] 
\\
&+
\OO_1^*(\eps\log \eps)\\
 h^*=&\ h - \ \ \eps M_1^{\sigma,\tau}
+\ \ \eps^2 M_2^{\sigma,\tau}+\ \ \OO_3^*(\eps)\\
 \tau^*=&\ \tau+\  \ol t +\qquad \qquad 
\pa_h w(\eta^*,h^*) 
\left[\log| w(\eta^*,h^*)| + (\Phi^\sg)' ( w (\eta^*,h^*))
\right] 
\\
&+ \OO_1^*(\eps\log\eps),
  \end{split}
 \]
where $M^*_i$ and $\Phi^\sigma$ are $\CCC^\infty$ functions defined in 
Lemmas \ref{lemma:GluingUnperturb} and \ref{lemma:SM:SR:1}  
respectively, and $\bar t$ is an integer satisfying 
(\ref{eq:time-interval}). The functions $M_i^*$ are evaluated at 
$(\eta^*, \xi,h^*,\tau)$ and satisfy
\[
 M_1^{\sigma, 1}=\pa_\xi\Theta^\sigma +\OO^*_2(w),
\qquad M_1^{\sigma, 2}=\pa_\tau\Theta^\sigma+\OO^*_2(w), 
\]
where $\Theta^\sigma$ is the Melnikov potential defined by 
\begin{equation}\label{def:Melnikov:Arnold}
\Theta^\sigma (\eta,\xi,\tau)=\int_{-\infty}^{+\infty} \left(H_1\left(\Ga^\sigma
(\eta,\xi,\tau+t),t-\tau\right)-H_1\left(\eta,\varphi+\eta 
t,0,0,t-\tau\right)\right)\,dt
\end{equation}
and $\Ga^\sigma$ are the time parameterization of the pendulum separatrices,
that is 
\[
\Ga^\sigma (\eta,\xi,\tau)=\left(\eta, \xi+\eta\tau, 4\arctan
(e^{\sigma\tau}),\frac{2\sigma}{\cosh \tau}\right).
\]
\end{theorem}

\begin{proof}This theorem can be easily deduced from Theorem \ref{thm:SM:SR}.
First, one  needs to recall that the Moser normal form only depends on $x$ 
and $y$
(this is the reason for the better estimates for the function $f$).

To deduce the formula of the Melnikov potentials $\Theta^\sigma$ defined in
\eqref{def:Melnikov} it is enough to recall that for the generalized Arnold
model \eqref{def:ArnoldGeneralized}, $\mu^\sigma=0$. This implies that the 
$\vartheta^\sigma$ contribution to \eqref{def:Melnikov} vanishes and therefore 
the splitting potential is just given by the Melnikov integral. Since $H_0$ is 
uncoupled we also have $\chi^\pm=0$. Then, the Melnikov 
potential can be just written as one integral.
\end{proof}

Now we analyze the resonant regions. We consider the slow-fast variables 
$(J,\theta)$ 
defined in Section
\ref{sec:finiteharmonics}.
The function $w_0$ defined in \eqref{def:omega0} just becomes
\[
 w_0(\eta^*,h^*,\xi^*)=h^*-\frac{(\eta^*)^2}{2}-\eps {\bf 
H}_1(\eta^*,\xi^*).
\]

\begin{theorem}\label{thm:SM:DR:Arnold}
Fix $\beta>0$,  $1\ge a>0$ and $k\in\NNN(H_1)\cup\NNN^{(2)}(H_1)$. 
For $\eps$ sufficiently small there exist $c>0$ independent 
of $\eps$ and canonical coordinates $(\eta,\xi,h,\tau)$ 
such that in the resonant zone $ \Dr^k_\beta$ 
the following conditions hold: 
\begin{itemize}
\item the canonical form $\om=d\eta\wedge d\xi+ 
dh\wedge d\tau$;
\item 
$\eta = I +\OO^*_1(\eps,H_0-E(I)), \,
\xi + \eta\tau=\varphi+f, \, 
h=H_0+\OO^*_1(\eps,H_0-E(I)),$
where $f$ denotes a function depending only on 
$(I,p,q,\eps)$ and such that 
$f=\mathcal O(\eps)$. 

\item In these coordinates 
$\SM_\eps$ has the following form. For any 
$\sigma\in \{-,+\}$ and $(\eta^*,h^*)$ such that 
\[
 c\ii\eps^{1+a}<|w_0(\eta^*,h^*,\xi^*)|<c\eps,\quad |\tau|<c\ii, \quad
c<|w_0(\eta^*,h^*,\xi^*)|\,e^{\ol t}<c\ii,
\]
the separatrix map $(\eta^*, \xi^*,h^*,\tau^*)=
\SM_\eps(\eta, \xi,h,\tau)$ is defined implicitly as follows
\begin{equation*}
\begin{aligned}
\eta^*=&\eta&&+\eps\pa_\xi\Theta^\sigma(\eta,\xi,\tau)\quad  &&+\eps 
B^{\eta,\sigma} (\eta,\xi, w_0,\tau)&& +\OO^*(\eps^{5/3})\\
 \xi^*=&\xi &&+
 \eta\log\left|\kk^\sigma w_0\right|
&&+\eps B^{\xi,\sigma} (\eta,\xi, w_0,\tau, \ol t) &&+\OO^*(\eps)\\
 h^*=&h &&+\eps\pa_\tau\Theta^\sigma(\eta,\xi,\tau)&& 
 \qquad \qquad \quad +\OO^*(\eps^{5/3})\\
 \tau^*=&\tau+\ol t&& +
 \log\left|\kk^\sigma w_0\right|&& 
 +\eps B^{\tau,\sigma}(\eta,\xi,\tau, w_0,\ol t)&&+\OO^*(\eps)\\
\sigma^*=&\sigma\ \mathrm{sgn}\ w_0
\end{aligned}
\end{equation*}
where $\Theta^\sigma$ are the Melnikov potentials given in
\eqref{def:Melnikov:Arnold}. If we define, $z=(\eta,\xi, w_0,\tau)$ the 
functions  $B^{\eta,\sigma}$, $B^{\tau,\sigma}$ and $B^{\xi,\sigma}$ are 
defined by
\[
 B^{\eta,\sigma} (z)= 
\frac{1}{\eta}\left({\bf\bar
H}_1(\eta,\xi)-{\bf\bar
H}_1(\eta,\xi+\eta
\log\left|\kk^\sigma w_0\right|)\right)
\]
and 
\[
 \begin{split}
 B^{\tau,\sigma} (z,\bar t)&= 
f_1(z)\tau+
f_2(z)\log \left|\kk^\sigma w_0\right|+
f_3(z, \ol t)\,\ol t\\
 B^{\xi,\sigma} (z,\bar t)&=  
g_1(z)\tau+
g_2(z)\log \left|\kk^\sigma w_0\right|+
g_3(z, \ol t)\ \ol t
\end{split}
\]
for certain $\CCC^\infty$ functions $f_i$, $g_i$ satisfying $f_i,g_i=\OO^*(1)$.
\end{itemize}

\end{theorem}

\begin{proof}
To prove this theorem it is enough to use the results in Theorem 
\ref{thm:SM:DR} and take into account that the particular form of $H_0$ implies 
that $\la=1$ and  $\mu^\sigma=0$. Then, one 
can easily see that $B^{h,\sigma}$ and $B^{\eta,\sigma}$ have the 
form given in this theorem. Moreover, reasoning as in the proof of Theorem 
\ref{thm:SM:SR:Arnold}, one can see that the splitting potential is given by 
formula  \eqref{def:Melnikov:Arnold}.


\end{proof}

\bibliography{references}

\def\cprime{$'$} \def\cprime{$'$} \def\cprime{$'$}
\begin{thebibliography}{BdlLW96}

\bibitem[Arn64]{Arnold64}
V.I. Arnold.
\newblock Instability of dynamical systems with several degrees of freedom.
\newblock {\em Sov. Math. Doklady}, 5:581--585, 1964.

\bibitem[BdlLW96]{BanyagaLW}
A.~Banyaga, R.~de~la Llave, and C.E. Wayne.
\newblock Cohomology equations near hyperbolic points and geometric versions of
  {S}ternberg {L}inearization {T}heorem.
\newblock {\em J. of Geometric Analysis}, 6(4):613--649, 1996.

\bibitem[Ber08]{Be}
P.~Bernard.
\newblock The dynamics of pseudographs in convex {H}amiltonian systems.
\newblock {\em J. Amer. Math. Soc.}, 21(3):615--669, 2008.

\bibitem[BK05]{BourgainK04}
J.~Bourgain and V.~Kaloshin.
\newblock Diffusion for {H}amiltonian perturbations of integrable systems in
  high dimensions.
\newblock {\em J. Funct. Anal.}, (229):1--61, 2005.

\bibitem[BKZ11]{BernardKZ11}
P.~Bernard, V.~Kaloshin, and K.~Zhang.
\newblock Arnold diffusion in arbitrary degrees of freedom and crumpled
  3-dimensional normally hyperbolic invariant cylinders.
\newblock Preprint available at \url{http://arxiv.org/abs/1112.2773}, 2011.

\bibitem[BT99]{BolotinT99}
S.~Bolotin and D.~Treschev.
\newblock Unbounded growth of energy in nonautonomous {H}amiltonian systems.
\newblock {\em Nonlinearity}, 12(2):365--388, 1999.

\bibitem[Che13]{Cheng13}
C.Q. Cheng.
\newblock Arnold diffusion in nearly integrable hamiltonian systems.
\newblock 2013.

\bibitem[Chi79]{Chirikov79}
B.V. Chirikov.
\newblock A universal instability of many-dimensional oscillator systems.
\newblock {\em Phys. Rep.}, 52(5):264--379, 1979.

\bibitem[CK15]{CastejonK}
O.~Castejon and V.~Kaloshin.
\newblock Random iteration of maps of a cylinder and diffusive behavior.
\newblock Preprint available at \url{http://arxiv.org/abs/1501.03319}, 2015.

\bibitem[CY04]{ChengY04}
C.Q. Cheng and J.~Yan.
\newblock Existence of diffusion orbits in a priori unstable {H}amiltonian
  systems.
\newblock {\em J. Differential Geom.}, 67(3):457--517, 2004.

\bibitem[CY09]{ChengY09}
C.Q. Cheng and J.~Yan.
\newblock Arnold diffusion in hamiltonian systems: apriori unstable case.
\newblock {\em J. Differential Geom.}, 82:229--277, 2009.

\bibitem[CZ13]{ChengZ13}
C.~Q. Cheng and J.~Zhang.
\newblock Asymptotic trajectories of {KAM} torus.
\newblock Preprint available at \url{http://arxiv.org/abs/1312.2102}, 2013.

\bibitem[DdlLS00]{DelshamsLS00}
A.~Delshams, R.~de~la Llave, and T.M. Seara.
\newblock A geometric approach to the existence of orbits with unbounded energy
  in generic periodic perturbations by a potential of generic geodesic flows of
  $\mathbb{T}\sp 2$.
\newblock {\em Comm. Math. Phys.}, 209(2):353--392, 2000.

\bibitem[DdlLS06]{DelshamsLS06a}
A.~Delshams, R.~de~la Llave, and T.M. Seara.
\newblock A geometric mechanism for diffusion in hamiltonian systems overcoming
  the large gap problem: heuristics and rigorous verification on a model.
\newblock {\em Mem. Amer. Math. Soc.}, 2006.

\bibitem[DdlLS08]{DelshamsLS08}
A.~Delshams, R.~de~la Llave, and T.~M. Seara.
\newblock Geometric properties of the scattering map of a normally hyperbolic
  invariant manifold.
\newblock {\em Adv. Math.}, 217(3):1096--1153, 2008.

\bibitem[DdlLS13]{DelshamsLS13}
A.~Delshams, R.~de~la Llave, and T.~M. Seara.
\newblock Instability of high dimensional hamiltonian systems: Multiple
  resonances do not impede diffusion.
\newblock 2013.

\bibitem[DG00]{DelshamsG00}
A.~Delshams and P.~Guti{\'e}rrez.
\newblock Splitting potential and the {P}oincar{\'e}-{M}elnikov method for
  whiskered tori in {H}amiltonian systems.
\newblock {\em J. Nonlinear Sci.}, 10(4):433--476, 2000.

\bibitem[DH09]{DelshamsH09}
A.~Delshams and G.~Huguet.
\newblock Geography of resonances and {A}rnold diffusion in a priori unstable
  {H}amiltonian systems.
\newblock {\em Nonlinearity}, 22(8):1997--2077, 2009.

\bibitem[Eth05]{EK}
T.~Ethier, S.~Kurtz.
\newblock Markov processes: Characterization and convergence.
\newblock 2005.

\bibitem[FGKP11]{FejozGKR}
J.~Fejoz, M.~Guardia, V.~Kaloshin, and Roldan. P.
\newblock Kikrwood gaps and diffusion along mean motion resonance for the
  restricted planar three body problem.
\newblock Preprint available at \url{http://arXiv:1109.2892}, 2011.

\bibitem[GdlL06]{GideaL06}
M.~Gidea and R.~de~la Llave.
\newblock Topological methods in the instability problem of {H}amiltonian
  systems.
\newblock {\em Discrete Contin. Dyn. Syst.}, 14(2):295--328, 2006.

\bibitem[GK14]{GuardiaK14}
M.~Guardia and V.~Kaloshin.
\newblock Orbits of nearly integrable systems accumulating to {KAM} tori.
\newblock Preprint available at \url{http://arxiv.org/abs/1412.7088}, 2014.

\bibitem[GK15]{GuardiaK15}
M.~Guardia and V.~Kaloshin.
\newblock Stochastic diffusive behavior through big gaps in a priori unstable
  systems.
\newblock In preparation, 2015.

\bibitem[GT08]{GelfreichT08}
V.~Gelfreich and D.~Turaev.
\newblock Unbounded energy growth in {H}amiltonian systems with a slowly
  varying parameter.
\newblock {\em Comm. Math. Phys.}, 283(3):769--794, 2008.

\bibitem[Kal03]{Kaloshin03}
V.~Kaloshin.
\newblock Geometric proofs of mather's accelerating and connecting theorems.
\newblock {\em London Mathematical Society, Lecture Notes Series, Cambridge
  University Press}, pages 81--106, 2003.

\bibitem[KL08a]{KaloshinL08}
V.~Kaloshin and M.~Levi.
\newblock An example of {A}rnold diffusion for near-integrable {H}amiltonians.
\newblock {\em Bull. Amer. Math. Soc. (N.S.)}, 45(3):409--427, 2008.

\bibitem[KL08b]{KaloshinL08a}
V.~Kaloshin and M.~Levi.
\newblock Geometry of {A}rnold diffusion.
\newblock {\em SIAM Rev.}, 50(4):702--720, 2008.

\bibitem[KLS14]{Kaloshin-Levi-Saprykina}
V.~Kaloshin, M.~Levi, and M.~Saprykina.
\newblock Arnold diffusion in a pendulum lattice.
\newblock {\em Comm in Pure and Applied Math}, 67(5):748--775, 2014.

\bibitem[KMV04]{KaloshinMV04}
V.~Kaloshin, J.~Mather, and E.~Valdinoci.
\newblock Instability of totally elliptic points of symplectic maps in
  dimension 4.
\newblock {\em Ast{\'e}risque}, 297:79--116, 2004.

\bibitem[KS12]{KaloshinS12}
V.~Kaloshin and M.~Saprykina.
\newblock An example of a nearly integrable {H}amiltonian system with a
  trajectory dense in a set of maximal {H}ausdorff dimension.
\newblock {\em Comm. Math. Phys.}, 315(3):643--697, 2012.

\bibitem[KZ12]{KaloshinZ12}
V.~Kaloshin and K.~Zhang.
\newblock A strong form of {A}rnold diffusion for two and a half degrees of
  freedom.
\newblock Preprint available at
  \url{http://www.terpconnect.umd.edu/~vkaloshi/}, 2012.

\bibitem[KZZ15]{KaloshinZZ}
V.~Kaloshin, J.~Zhang, and K.~Zhang.
\newblock Normally {H}yperbolic {I}nvariant {L}aminations and diffusive
  behavior for the generalized {A}rnold example away from resonances.
\newblock Preprint available at
  \url{http://www.terpconnect.umd.edu/~vkaloshi/}, 2015.

\bibitem[Mat91a]{Mather91}
J.~N. Mather.
\newblock Action minimizing invariant measures for positive definite
  {L}agrangian systems.
\newblock {\em Math. Z.}, 207(2):169--207, 1991.

\bibitem[Mat91b]{Mather91a}
J.~N. Mather.
\newblock Variational construction of orbits of twist diffeomorphisms.
\newblock {\em J. Amer. Math. Soc.}, 4(2):207--263, 1991.

\bibitem[Mat93]{Mather93}
J.~N. Mather.
\newblock Variational construction of connecting orbits.
\newblock {\em Ann. Inst. Fourier (Grenoble)}, 43(5):1349--1386, 1993.

\bibitem[Mat96]{Mather96}
J.~N. Mather.
\newblock Manuscript.
\newblock Unpublished, 1996.

\bibitem[Mat03]{Mather03}
J.~N. Mather.
\newblock Arnold diffusion. {I}. {A}nnouncement of results.
\newblock {\em Sovrem. Mat. Fundam. Napravl.}, 2:116--130 (electronic), 2003.

\bibitem[Mat08]{Mather08}
J.N. Mather.
\newblock Arnold diffusion {II}.
\newblock Preprint, 185pp, 2008.

\bibitem[Moe96]{Moeckel:1995}
R.~Moeckel.
\newblock Transition tori in the five-body problem.
\newblock {\em J. Differential Equations}, 129(2):290--314, 1996.

\bibitem[Mos56]{Moser56}
J.~Moser.
\newblock The analytic invariants of an area-preserving mapping near a
  hyperbolic fixed point.
\newblock {\em Comm. Pure Appl. Math.}, 9:673--692, 1956.

\bibitem[MS02]{MarcoS03}
J.P. Marco and D.~Sauzin.
\newblock Stability and instability for {G}evrey quasi-convex near-integrable
  {H}amiltonian systems.
\newblock {\em Publ. Math. Inst. Hautes {\'E}tudes Sci.}, (96):199--275 (2003),
  2002.

\bibitem[MS04]{MarcoS04}
J.P. Marco and D.~Sauzin.
\newblock Wandering domains and random walks in gevrey near integrable systems.
\newblock {\em Ergodic Theory \& Dynamical Systems}, (24, 5):1619 -- 1666,
  2004.

\bibitem[Pif06]{Pift06}
G.~Piftankin.
\newblock Diffusion speed in the {M}ather problem.
\newblock {\em Nonlinearity}, (19):2617--2644, 2006.

\bibitem[PT07]{PT}
G.~N. Piftankin and D.~V. Treshch{\"e}v.
\newblock Separatrix maps in {H}amiltonian systems.
\newblock {\em Uspekhi Mat. Nauk}, 62(2(374)):3--108, 2007.

\bibitem[Sau06]{Sauzin06}
D.~Sauzin.
\newblock Exemples de diffusion d'{A}rnold avec convergence vers un mouvement
  brownien.
\newblock Preprint, 2006.

\bibitem[{\v{S}}il65]{Shilnikov}
L.~P. {\v{S}}il{\cprime}nikov.
\newblock A case of the existence of a denumerable set of periodic motions.
\newblock {\em Dokl. Akad. Nauk SSSR}, 160:558--561, 1965.

\bibitem[Tre98]{Treschev98}
D.~Treschev.
\newblock Width of stochastic layers in near-integrable two-dimensional
  symplectic maps.
\newblock {\em Phys. D}, 116(1-2):21--43, 1998.

\bibitem[Tre02]{Treschev02a}
D.~Treschev.
\newblock Multidimensional symplectic separatrix maps.
\newblock {\em J. Nonlinear Sci.}, 12(1):27--58, 2002.

\bibitem[Tre04]{Treschev04}
D.~Treschev.
\newblock Evolution of slow variables in a priori unstable hamiltonian systems.
\newblock {\em Nonlinearity}, 17(5):1803--1841, 2004.

\bibitem[Tre12]{Treschev12}
D.~Treschev.
\newblock Arnold diffusion far from strong resonances in multidimensional a
  priori unstable hamiltonian systems.
\newblock {\em Nonlinearity}, (9):2717--2757, 2012.

\bibitem[ZF68]{Zaslavskii}
G.M. Zaslavskii and N.~N. Filonenko.
\newblock Stochastic instability of trapped particles and conditions of
  applicability of the quasi-linear approximation.
\newblock {\em Soviet Phys. JETP}, 27:851--857, 1968.

\end{thebibliography}
\bibliographystyle{alpha}

\end{document}